\RequirePackage{fix-cm}
\documentclass[smallextended]{svjour3}       
\usepackage{graphicx}

\usepackage{relsize}

\usepackage{amsmath,amsthm,amssymb, bbm}
\begin{document}

\title{Total variation approximation of random orthogonal matrices by Gaussian matrices \thanks{Supported in part by NSF grant DMS 1612589 to Elizabeth Meckes.}}

\titlerunning{Approximation of random matrices }        

\author{Kathryn Stewart}


\institute{Kathryn Stewart \at
              Department of Mathematics, Case Western Reserve University, 231 Yost Hall, Cleveland, OH 44106 USA \\
              \email{kathrynstewart@case.edu}           
}

\date{Received: date / Accepted: date}

\maketitle

\begin{abstract}
The topic of this paper is the asymptotic distribution of the entries of random orthogonal matrices distributed according to Haar measure. We examine the total variation distance between the joint distribution of the entries of $W_n$, the $p_n \times q_n$ upper-left block of a Haar-distributed matrix, and that of $p_nq_n$ independent standard Gaussian random variables, and show that the total variation distance converges to zero when $p_nq_n = o(n)$. 
\end{abstract}

\section{Introduction}
\label{intro}
Let $U_n$ be a random orthogonal matrix which is distributed according to Haar measure on the orthogonal group $\mathcal{O}(n)$. The asymptotic distribution of the individual entries of such a Haar-distributed matrix is classical. Borel \cite{Borel} showed in 1906 that a single coordinate of a randomly chosen point on the sphere is asymptotically Gaussian. That is, if $X = (X_1, \cdots, X_n)$ is a uniform random vector in 
\begin{equation*}
\mathbb{S}^{n-1}= \{ x \in \mathbb{R}^n : ||x||=1 \},
\end{equation*}
then for all $t \in \mathbb{R}$
\begin{equation*}
\mathbb{P}[ \sqrt{n}X_1 \leq t] \xrightarrow{n \to \infty} \frac{1}{\sqrt{2 \pi}} \int_{- \infty}^t e^{-x^2/2} dx.
\end{equation*} 
It follows by one of the standard constructions of Haar measure that the sequence $\{ \sqrt{n}[U_n]_{1,1} \}$ converges weakly to the standard Gaussian distribution as $n \to \infty$. By symmetry, this means that all of the individual entries of a random  orthogonal matrix are approximately Gaussian, for large matrices. 

Diaconis and Freedman \cite{DiaFre} gave a substantial strengthening of Borel's result, showing that the joint distribution of the first $k$ coordinates of a uniform random point on the sphere is close in total variation distance to $k$ independent identically distributed Gaussian random variables if $k = o(n)$, as follows.

\begin{theorem}[Diaconis-Freedman] 
Let $X$ be a uniform random point on $\sqrt{n} \mathbb{S}^{n-1}$, for $n \geq 5$, and let $1 \leq k \leq n-4$. Let $Z$ be a standard Gaussian random vector in $\mathbb{R}^k$. Then the total variation distance between the distribution of the first $k$ coordinates of $X$ and the distribution of $Z$ is 
\begin{equation*}
d_{TV}((X_1, \cdots, X_k),Z) \leq \frac{2(k+3)}{n-k-3}.
\end{equation*}
\end{theorem}

The theorem implies that for $k = o(n)$, one can approximate any $k$ entries from the same row or column of a uniform random orthogonal matrix $U_n$ by independent Gaussian random variables. This led Diaconis to consider the question of how many entries of $U_n$ can be simultaneously approximated by independent normal random variables. A sufficient condition was given by Diaconis, Eaton, and Lauritzen \cite{DiaEatLau}, which was improved by Tiefeng Jiang \cite{Jiang} to the following:

\begin{theorem}[Jiang]
Let $\{p_n : n \geq 1 \}$ and $\{ q_n : n \geq 1 \}$ be two sequences of positive integers such that $p_n = o(\sqrt{n})$ and $q_n = o(\sqrt{n})$ as $n \to \infty$. For each $n$, let $U_n$ be a random matrix uniformly distributed on the orthogonal group $\mathcal{O}(n)$ and suppose that $W_n$ is the $p_n \times q_n$ upper left block of $U_n$. Let $X_n$ be a $p_n \times q_n$ matrix of independent identically distributed standard Gaussian random variables, then 
\begin{equation*}
\lim_{n \to \infty} d_{TV}\left( \sqrt{n}W_n, X_n \right) = 0.
\end{equation*}
\end{theorem}
Jiang further showed that the theorem was sharp for square submatrices; that is, there are $x > 0$ and $y>0$ such that $p_n \sim x\sqrt{n}$ and $q_n \sim y\sqrt{n}$ and 
\begin{equation*}
\liminf_{n \to \infty} d_{TV} \left( \sqrt{n}W_n, X_n \right) \geq \phi (x,y) > 0
\end{equation*}
where $\phi(x,y) := \mathbb{E} | \exp(-\frac{x^2y^2}{8} + \frac{xy}{4}\xi) - 1| \in (0,1)$ and $\xi$ is a standard normal. 

Jiang also showed in \cite{Jiang} that relaxing the sense in which the entries of the random matrix should be simultaneously approximable by independent identically distributed Gaussian variables allows a larger collection of entries to be approximated.

Specifically, if $U_n=[u_{ij}]_{i,j=1}^n$ is an orthogonal matrix obtained from performing the Gram-Schmidt procedure on a matrix $Y_n = [y_{ij}]_{i,j=1}^n$ whose elements are independent standard normals, then Jiang proved that the maximum order of $m_n$ such that 
\begin{equation*}
\max_{1 \leq i \leq n, 1 \leq j \leq m_n} |\sqrt{n}u_{ij} - y_{ij}| \to 0
\end{equation*} in probability is $m_n = o(n/ \log n)$. 

In \cite{ChaMec}, Chatterjee and Meckes showed that any subcollection of entries of size $o(n)$, not just those arising as principal submatrices, is approximately Gaussian. More generally, they showed that any projection of Haar measure is close to Gaussian, as long as the projection dimension is $o(n)$.

\begin{theorem}[Chatterjee-Meckes]
Let $A_1, \cdots, A_k$ be $n \times n$ matrices over $\mathbb{R}$ satisfying $tr(A_iA_j^T) = n\delta_{ij}$; that is, $\{ \frac{1}{\sqrt{n}}A_i \}_{1 \leq i \leq k}$ is orthonormal with  respect to the Hilbert-Schmidt inner product. Let $U_n$ be a random matrix uniformly distributed on $\mathcal{O}(n)$, and consider the random vector 
\begin{equation*}
X = (tr(A_1U_n), tr(A_2U_n), \cdots, tr(A_kU_n))
\end{equation*}
in $\mathbb{R}^k$. Let $Z = (Z_1, \cdots, Z_k)$ be a random vector whose components are independent standard normal random variables. Then for $n \geq 2$,
\begin{equation*}
W_1(X,Z) \leq \frac{\sqrt{2}k}{n-1},
\end{equation*}
where $W_1(\cdot, \cdot)$ denotes the $L_1-$Wasserstein distance between distributions.
\end{theorem} 

This result, together with the results of Jiang and Diaconis-Freedman, suggest that one should be able to approximate the top left $p_n \times q_n$ block by independent identically distributed Gaussian random variables, in total variation distance, as long as $p_nq_n = o(n)$. The following theorem verifies this conjecture.

\begin{theorem} \label{T:Main}
Let $\{p_n : n \geq 1 \}$ and $\{ q_n : n \geq 1 \}$ be two sequences of positive integers such that $p_nq_n = o(n)$ as $n \to \infty$. For each $n$, let $U_n$ be a random matrix uniformly distributed on the orthogonal group $\mathcal{O}(n)$ and suppose that $W_n$ is the $p_n \times q_n$ upper left block of $U_n$. Let $X_n$ be a $p_n \times q_n$ matrix of independent identically distributed standard Gaussian random variables, then 
\begin{equation*}
\lim_{n \to \infty} d_{TV}\left( \sqrt{n}W_n, X_n \right) = 0.
\end{equation*}
\end{theorem}
This theorem thus unifies Jiang's result and the Diaconis-Freedman result. 

It should be noted that the main theorem presented here also appears in the independent simultaneous work \cite{JiaMa} by T. Jiang and Y. Ma. The approach in \cite{JiaMa} relates the total variation distance to the Kullback-Leibler distance and then shows convergence in the Kullback-Leibler distance.

The approach here works directly with the total variation distance and is an extension of the approach taken in \cite{Jiang}, however the analysis is much more delicate if the only assumption is that $p_nq_n = o(n)$. In particular, the proof requires sharp asymptotics for the covariances of traces of powers of Wishart matrices, for powers growing with the size of the matrix. Bai \cite{Bai} has developed asymptotics for the expected value of traces of powers of Wishart matrices using graph theory and combinatorics. We give an extension of Bai's result, as well as providing sharp asymptotics for the covariances, which may be of independent interest. 

The contents of this paper are as follows. In section 2, we give the proof of the main theorem, making use of new estimates on the asymptotic means and covariances of traces of powers of Wishart matrices. Section 3 contains the proofs of these asymptotics; some technical estimates used in section 2 are relegated to the appendix. 
\section{Proof of the Main Theorem}
\label{sec: Main Proof}
Let $\mu$ and $\nu$ be probability measures on $( \mathbb{R}^m, \mathcal{B})$, where $\mathcal{B}$ is the Borel $\sigma$-algebra. The total variation distance between $\mu$ and $\nu$ is 
\begin{equation*}
d_{TV} \left( \mu, \nu \right) = 2\sup_{A \in \mathcal{B}} |\mu(A) - \nu(A)|. 
\end{equation*}
If $\mu$ and $\nu$ have densities $f(x)$ and $g(x)$ with respect to Lebesgue measure, then 
\begin{equation*}
d_{TV} \left( \mu, \nu \right) = \int_{\mathbb{R}^m} |f(x) - g(x)|dx_1dx_2 \cdots dx_m.
\end{equation*}

Let $f_n(z)$ be the joint density function of $\sqrt{n}W_n$, the $p_n \times q_n$ upper left block of the random orthogonal matrix $\sqrt{n}U_n$. We will assume throughout that $q_n \leq p_n$. Let $X_n$ be a $p_n \times q_n$ matrix of independent identically distributed standard Gaussian random variables and let $g_n(z)$ denote the joint density of the entries of $X_n$. The total variation distance between the entries of $\sqrt{n}W_n$ and those of $X_n$ is
\begin{align*}
d_{TV} \left( \sqrt{n}W_n, X_n \right) & = \int_{\mathbb{R}^{p_nq_n}} \left| f_n(z) - g_n(z) \right| dz \\
& = \int_{\mathbb{R}^{p_nq_n}} \left| \frac{f_n(z)}{g_n(z)} - 1 \right| g_n(z) dz \\
& = \mathbb{E} \left| \frac{f_n(X_n)}{g_n(X_n)} - 1 \right|.
\end{align*}
The following formula for the joint density function $f_n(z)$ of the entries of $W_n$ is due to Eaton \cite{Eaton}.

\begin{theorem}[Eaton]
Let $U_n$ be an $n \times n$ random orthogonal matrix, and let $W_{p,q}$ denote the upper left $p \times q$ subblock of $U_n$. For $q \leq p$ and $p+q \leq n$, with probability one the random matrix $W_{p,q}$ lies in the set $\mathcal{X}$ of $p \times q$ matrices $X$ over $\mathbb{R}$ with the property that all of the eigenvalues of $X^TX$ lie in (0,1), and the density of $W_{p,q}$ with respect to Lebesgue measure on $\mathcal{X}$ is given by 
\begin{equation*}
f(z)= C_1 \det \left( I_q - z^Tz \right)^{\frac{n-p-q-1}{2}}I_0(z^Tz),
\end{equation*}
where the constant $C_1$ is 
\begin{equation*}
C_1 = \left( \sqrt{2 \pi} \right)^{-pq} \frac{\omega (n-p,q)}{\omega (n,q)},
\end{equation*}
with $\omega ( \cdot, \cdot )$ denoting the Wishart constant defined by
\begin{equation*}
\frac{1}{\omega (r, s)} = \pi^{\frac{s(s-1)}{4}} 2^{\frac{rs}{2}} \prod_{j=1}^s \Gamma \left( \frac{r-j+1}{2} \right)
\end{equation*}
and $I_0(z^Tz)$ is the indicator that all the eigenvalues of $z^Tz$ lie in (0,1). 
Here $s$ is a positive integer and $r$ is a real number, $r > s-1$. 
\end{theorem}
If follows that the density function of $\sqrt{n}W_n$ is
\begin{equation*}
f_n(z) = \left( \sqrt{2\pi n} \right)^{-pq} \frac{\omega (n-p,q)}{\omega (n,q)} \left[ \det \left( I_q - \frac{z^Tz}{n} \right)^{\frac{n-p-q-1}{2}} \right] I_0\left( \frac{z^Tz}{n} \right),
\end{equation*} 
where for notational convenience we use $p$ and $q$ in place of $p_n$ and $q_n$. The joint density function of $p_nq_n$ independent standard Gaussian random variables is
\begin{equation*}
g_n(z) = \left( \sqrt{2 \pi} \right)^{-pq} \exp \left( \frac{-tr(z^Tz)}{2} \right),
\end{equation*}
where $z$ is a $p_n$ by $q_n$ matrix. Let $\lambda_1, \cdots, \lambda_q$ be the eigenvalues of $X_n^TX_n$. Then the ratio $\frac{f_n(X_n)}{g_n(X_n)}$ can be written as a product of a constant  part $K_n$ and a random part $L_n$, where
\begin{align*}
& K_n = \left( \frac{2}{n} \right)^{\frac{pq}{2}} \prod_{j=1}^q \frac{\Gamma ((n-j+1)/2)}{\Gamma ((n-p-j+1)/2)} ; \\
& L_n = \left[ \prod_{i=1}^q \left (1-\frac{\lambda_i}{n} \right) \right] ^{\frac{n-p-q-1}{2}} \exp \left( \frac{1}{2} \sum_{i=1}^q \lambda_i \right) 
\end{align*}
if all the $\lambda_i$ are in $(0,n)$ and $L_n$ is zero otherwise.
Then 
\begin{equation*}
d_{TV} \left( \sqrt{n}W_n, X_n \right) = \mathbb{E} \left| \frac{f_n(X_n)}{g_n(X_n)} -1 \right| = \mathbb{E} |K_n \cdot L_n - 1|.
\end{equation*}
Note that $K_nL_n \geq 0$ and $\mathbb{E}[K_nL_n] = \int_{\mathbb{R}^{p_nq_n}} f_n(x) dx =1$. It is a standard exercise that these facts, together with the convergence \emph{in probability} of $K_nL_n$, suffice to show that the $\{ K_nL_n \}$ are uniformly integrable, which in turn gives the required convergence in expectation. 

Define a function  $F(x)$ by $F(x) = \frac{x}{2} + \frac{n-p-q-1}{2} \log(1- \frac{x}{n})$ if $0 \leq x <n$ and $F(x) = -\infty$ otherwise. Then $L_n = \exp (\sum_{i=1}^q F(\lambda_i))$, and showing that $K_nL_n \xrightarrow{\mathbb{P}} 1$ as $n \to \infty$ is equivalent to showing that
\begin{equation*}
 \log(K_n)+\log(L_n) = \log (K_n)+\left \{ \sum_{i=1}^q F(\lambda_i) \right \} \xrightarrow{\mathbb{P}} 0.
\end{equation*}
Take $l$ to be the smallest odd integer with $l \geq \frac{\log p}{\log( \frac{n}{pq})}$. Then using Taylor's Theorem to expand $\log(1 - \frac{x}{n})$ up to order $l$, for any $x \in (0,n)$, 
\begin{equation*}\begin{split}
\log \left( 1 - \frac{x}{n} \right) & = - \frac{x}{n} - \frac{x^2}{2n^2} - \cdots - \frac{x^l}{ln^l} - \frac{x^{l+1}}{(l+1)(\xi_x - n)^{l+1}},
\end{split}\end{equation*}
where $\xi_x \in (0,x)$. Then
\begin{equation*}\begin{split}
F(x) & = \frac{p+q+1}{2n}x - \frac{n-p-q-1}{4n^2} x^2 - \cdots - \frac{n-p-q-1}{2ln^l}x^l + \frac{a_n(x)}{n^l}x^{l+1}
\end{split}\end{equation*}
where $a_n(x) = \frac{-n^l(n-p-q-1)}{2(l+1)(\xi_x-n)^{l+1}}$. Since $X_n$ is a $p_n \times q_n$ matrix of independent standard Gaussian random variables, it follows from Lemma \ref{L: maxew} in the appendix that, with probability 1,
\begin{equation*}
\limsup \max_{1 \leq i \leq q_n} \frac{\lambda_i}{p_n} \leq 4.
\end{equation*} Fix $\epsilon > 0$ and define $\Omega_n = \left\{ \max_{1\leq i \leq q_n} \frac{\lambda_i}{p_n} \leq 4 + \epsilon  \right\}$, so that $\lim_{n \to \infty}\mathbb{P}(\Omega_n^C) = 0$. Then, on $\Omega_n$, 
\begin{align*}
 \log(L_n) & = \sum_{i=1}^q F(\lambda_i) \\
& = \sum_{i=1}^q \bigg[ \frac{p+q+1}{2n}\lambda_i - \frac{n-p-q-1}{4n^2} \lambda_i^2 \\
& \qquad - \cdots - \frac{n-p-q-1}{2ln^l}\lambda_i^l + \frac{a_n(\lambda_i)}{n^l}\lambda_i^{l+1} \bigg] \\
&  = \frac{(p+q+1)}{2n}tr(X^TX)  - \frac{(n-p-q-1)}{4n^2}tr(X^TX)^2 \\
& \qquad - \cdots - \frac{(n-p-q-1)}{2ln^{l}}tr(X^TX)^{l} +\sum_{i=1}^q \frac{a_n(\lambda_i) \lambda_i^{l+1}}{n^l} \\
& = \frac{1}{n} \left[ \frac{p+q+1}{2} tr(X^TX) - \frac{1}{4} tr(X^TX)^2 \right]  \\
& \qquad + \frac{1}{n^2} \left[ \frac{p+q+1}{4}tr(X^TX)^2 - \frac{1}{6}tr(X^TX)^3 \right] \\
& \qquad + \frac{1}{n^3} \left[ \frac{p+q+1}{6}tr(X^TX)^3 - \frac{1}{8}tr(X^TX)^4 \right] \\
& \qquad + \cdots  + \frac{p+q+1}{2ln^l}tr(X^TX)^l + \sum_{i=1}^q \frac{a_n(\lambda_i) \lambda_i^{l+1}}{n^l}.
\end{align*}
 Notice that 
\begin{equation*}\begin{split}
\bigg| \sum_{i=1}^q \frac{a_n(\lambda_i) \lambda_i^{l+1}}{n^l} \bigg| & \leq \sum_{i=1}^q \frac{|a_n(\lambda_i)|}{n^l} \lambda_i^{l+1} = \sum_{i=1}^q \frac{n^l (n-p-q-1)}{2(l+1)(\xi_i -n)^{l+1} n^l} \lambda_i^{l+1}.
\end{split}\end{equation*}
On $\Omega_n, |\lambda_i| \leq (4+\epsilon)p_n$. Then $|\xi_i| \leq (4+ \epsilon)p_n$ and $p_n = o(n)$, so that
\begin{equation*} \begin{split}
\bigg| \sum_{i=1}^q \frac{a_n(\lambda_i) \lambda_i^{l+1}}{n^l} \bigg| & \leq \sum_{i=1}^q \frac{n^l(n-p-q-1)}{2(l+1)n^l(n - (4+\epsilon)p_n)^{l+1}} \lambda_i^{l+1} \leq \frac{ \widetilde{a_n}tr(X^TX)^{l+1}}{n^l}
\end{split} \end{equation*}
 where $\widetilde{a_n}= \frac{n^l(n-p-q-1)}{2(l+1)(n-(4+\epsilon)p_n)^{l+1}}$ is bounded, independent of $n$. Define 
\begin{equation*}
h_i = \frac{1}{n^i} \left[ \frac{p+q+1}{2i}tr(X^TX)^i - \frac{1}{2(i+1)}tr(X^TX)^{i+1} \right]
\end{equation*}
when $i<l$ and $h_l = \frac{1}{n^l} [\frac{p+q+1}{2l}tr(X^TX)^l]$. Let $E_i = \mathbb{E}[h_i]$ and $R_i = h_i - E_i$. Finally, let $A= \frac{\widetilde{a_n}tr(X^TX)^{l+1}}{n^l}$, so 
\begin{equation*} \sum_{i=1}^q F(\lambda_i) = \sum_{i=1}^l R_i + \sum_{i=1}^l E_i + A;
\end{equation*}
 the goal is to show that 
\begin{equation*}
\left( \log(K_n)+\sum_{i=1}^l R_i + \sum_{i=1}^l E_i + A \right) \xrightarrow{\mathbb{P}} 0.
\end{equation*}
By Lemma \ref{L: Cancel} in the appendix, 
\begin{equation*}
\lim_{n \to \infty} \left( \log(K_n) + \sum_{i=1}^{l} E_i \right) = 0.
\end{equation*}   
To show $\left( \log(K_n) + \sum_{i=1}^q F(\lambda_i)\right) \xrightarrow{\mathbb{P}} 0$, it thus suffices to show that 
\begin{equation*}
\left( \sum_{i=1}^l R_i + A \right) \to 0
\end{equation*}
in probability as $n \to \infty$. Note that if $p=p_n$ is bounded and independent of $n$, then $l = 1$ for $n$ large enough. The only terms are then $R_1 =  \frac{1}{n} [\frac{p+q+1}{2}tr(X^TX)]$ and $A$.

First considering the sum of the $R_i$, fix $\epsilon > 0$ and recall $l$ is the smallest odd integer such that $l \geq \frac{\log p}{\log( \frac{n}{pq})}$. Define $\epsilon_i := \frac{\epsilon}{f(i)Z_l}$ where $f(i) = \big( \frac{n}{8e^3pq} \big)^{i-1}$ and $Z_l = \sum_{i=1}^l \frac{1}{f(i)}$, so that $\sum_{i=1}^l \epsilon_i = \epsilon$. Note that 
\begin{equation*}\begin{split}
\lim_{n \to \infty}Z_l & = \lim_{n \to \infty} \sum_{i=1}^l \left( \frac{8e^3pq}{n} \right)^{i-1} \\
& = \lim_{n \to \infty} \frac{ 1- \left( \frac{8e^3pq}{n} \right)^l}{1- \frac{8e^3pq}{n}} \\
& = \lim_{n \to \infty} \frac{1- \left( \frac{8e^3pq}{n} \right)^{\frac{\log p}{\log(\frac{n}{pq})}}}{1- \frac{8e^3pq}{n}} \\
& = \lim_{n \to \infty} \frac{1- e^{\frac{\log p}{\log(\frac{n}{pq})}\log \left( \frac{8e^3pq}{n} \right)}}{1- \frac{8e^3pq}{n}} \\
& = \lim_{n \to \infty} \frac{1- e^{\log p \left( -1 + \frac{\log 8e^3}{\log \left( \frac{n}{pq} \right)} \right) }}{1- \frac{8e^3pq}{n}} \\
& = 1,
\end{split}\end{equation*}
since $\frac{\log 8e^3}{\log \left( \frac{n}{pq} \right)} < \frac{1}{2}$ for $n$ large enough.
It follows from Chebychev's Inequality that
\begin{align*}
\mathbb{P} \left[  \bigg| \sum_{i=1}^{l-1} R_i \bigg| \geq \frac{\epsilon}{2} \right] & \leq \sum_{i=1}^{l-1} \mathbb{P} \left[  |R_i| \geq \frac{\epsilon_i}{2} \right] \leq 4\sum_{i=1}^{l-1} \frac{ Var[R_i]}{\epsilon_i^2} = 4\sum_{i=1}^{l-1} \frac{ Var[h_i]}{\epsilon_i^2} .
\end{align*}

Recall that the sum of the $R_i$ only occurs in the case that $p_n \to \infty$ as $n \to \infty$, in which case Lemma \ref{L: cov2} provides an explicit formula for the covariances. The variance $Var[R_i] = Var[h_i]$ of the individual terms is thus computed as follows,

\begin{align*}
Var &  \bigg[ \frac{p+q+1}{2in^i}tr(X^TX)^i - \frac{1}{2(i+1)n^i}tr(X^TX)^{i+1} \bigg]\\
& = \frac{(p+q+1)^2}{4i^2n^{2i}} Var[tr(X^TX)^i] + \frac{1}{4(i+1)^2n^{2i}}Var[tr(X^TX)^{i+1}]\\
& \qquad - \frac{2(p+q+1)}{4i(i+1)n^{2i}} Cov( tr(X^TX)^i, tr(X^TX)^{i+1})\\
& = \tfrac{(p+q+1)^2}{4i^2n^{2i}} \left( 2i^2((p)_i(p)_{i-1}q+p(q)_i(q)_{i+1}) + Ce_{i,i}+ Df_{i,i} \right) \\
& \qquad + \tfrac{1}{4(i+1)^2n^{2i}}\bigg( 2(i+1)^2((p)_{i+1}(p)_iq+p(q)_{i+1}(q)_i) \\
& \qquad + Ce_{i+1,i+1} + Df_{i+1,i+1} \bigg) \\
& \qquad - \tfrac{2(p+q+1)}{4i(i+1)n^{2i}} \left( 2i(i+1)((p)_i(p)_iq+p(q)_i(q)_i) + Ce_{i,i+1}+ Df_{i,i+1} \right)\\
& = \tfrac{(p)_{i-1}(p)_iq \left(q^2 - p + i^2 + i +2qi - 1 \right) + p(q)_{i-1}(q)_i \left( p^2 - q + i^2 + i + 2pi -1 \right) }{2n^{2i}} \\
& \qquad +  \tfrac{(p+q+1)^2}{4i^2n^{2i}} \left( Ce_{i,i}+ Df_{i,i} \right) + \tfrac{1}{4(i+1)^2n^{2i}}\left( Ce_{i+1,i+1} + Df_{i+1,i+1} \right) \\
& \qquad - \tfrac{2(p+q+1)}{4i(i+1)n^{2i}} \left( Ce_{i,i+1}+ Df_{i,i+1} \right) \\
& \leq \tfrac{p^{2i-1}q \left(q^2 - p + i^2 + i +2qi - 1 \right) + pq^{2i-1} \left( p^2 - q + i^2 + i + 2pi -1 \right) }{2n^{2i}} \\
& \qquad +  \tfrac{(p+q+1)^2}{4i^2n^{2i}} \left( Ce_{i,i}+ Df_{i,i} \right) + \tfrac{1}{4(i+1)^2n^{2i}}\left( Ce_{i+1,i+1} + Df_{i+1,i+1} \right) \\
& \qquad - \tfrac{2(p+q+1)}{4i(i+1)n^{2i}} \left( Ce_{i,i+1}+ Df_{i,i+1} \right),
\end{align*}
where
\begin{align*}
e_{i,i} & = o(ip^{2i-2}q^24^{2i}) \\
f_{i,i} & = o(i^8p^{2i-2}q^24^{2i}) \\
e_{i+1,i+1} & = o((i+1)p^{2i}q^24^{2i+2}) \\
f_{i+1,i+1} & = o((i+1)^8p^{2i}q^24^{2i+2}) \\
e_{i,i+1} & = o((i+1)p^{2i-1}q^24^{2i+1}) \\
f_{i,i+1} & = o((i+1)^8p^{2i-1}q^24^{2i+1}).
\end{align*}
Since $i \geq 1$, $(i+1)^6 \leq e^{6i}$. Thus
\begin{align*}
Var &  \bigg[ \frac{p+q+1}{2in^i}tr(X^TX)^i - \frac{1}{2(i+1)n^i}tr(X^TX)^{i+1} \bigg]\\
& \leq \frac{\widetilde{C}p^{2i}q^{2}4^{2i}e^{6i}}{in^{2i}}.
\end{align*}
Now, 
\begin{equation*} \begin{split} 
\sum_{i=1}^{l-1} \frac{(4e^3)^{2i}p^{2i}q^2}{in^{2i}}\frac{1}{\epsilon_i^2} & = \frac{1}{\epsilon^2} \sum_{i=1}^{l-1} \frac{ (4e^3p)^{2i}q^2Z_l^2n^{2i-2}}{in^{2i}(8e^3pq)^{2i-2}} \\
& = \frac{1}{\epsilon^2} \sum_{i=1}^{l-1} \frac{ (4e^3p)^{2}q^2Z_l^2}{in^{2}(2q)^{2i-2}}\\
& = \frac{16e^6Z_l^2}{\epsilon^2} \left( \frac{pq}{n} \right)^2 \sum_{i=1}^{l-1} \frac{1}{i(2q)^{2i-2}} \\
& \to 0
\end{split} \end{equation*}
as $n \to \infty$.
Therefore $\sum_{i=1}^{l-1} \frac{Var[h_i]}{\epsilon_i^2} \to 0$. 

The regrouping of the terms of $\log (L_n)$ terminates with 
\begin{equation*} h_l = \tfrac{(p+q+1)tr(X^TX)^l}{2ln^l}.\end{equation*}  The probability that $h_l \geq \frac{\epsilon}{4}$ is given by

\begin{equation*} \begin{split} 
\mathbb{P} & \left[ \frac{p+q+1}{2ln^l}tr(X^TX)^l \geq \frac{\epsilon}{4} \right] \\
& \qquad \leq \frac{4(p+q+1)^2}{\epsilon^2l^2n^{2l}}Var[tr(X^TX)^l] \\
& \qquad = \frac{4(p+q+1)^2}{\epsilon^2l^2n^{2l}} \left( 2l^2((p)_l(p)_{l-1}q + p(q)_{l-1}(q)_l) + Ce_{l,l} +Df_{l,l} \right)\\
& \qquad \leq \frac{4(p+q+1)^2}{\epsilon^2l^2n^{2l}} \left( 2l^2p^{2l-1}q + 2l^2pq^{2l-1} + \widetilde{C}p^{2l-2}q^24^{2l}l^8 \right).
\end{split} \end{equation*}
Now $l^6 \leq p$ for large enough $n$. Therefore, for large $n$,
\begin{equation*} \begin{split}
\mathbb{P}\left[ \frac{p+q+1}{2ln^l}tr(X^TX)^l \geq \frac{\epsilon}{4} \right] & \leq \frac{\widetilde{C}p^{2l+1}q^24^{2l}}{n^{2l}}.
\end{split} \end{equation*}
Now
\begin{equation} \begin{split} \label{eq: 2}
\frac{p^{2l+1}q^24^{2l}}{n^{2l}}& \leq \exp \left[ -2l \log \left( \frac{n}{pq} \right) + \log p + (2l) \log 4 \right] \\
& \leq \exp \left[ -2\frac{\log p}{\log( \frac{n}{pq} )}\log \left( \frac{n}{pq} \right) + \log p +2\frac{\log p}{\log( \frac{n}{pq} )}\log 4 \right] \\
& = \exp \left[ - \log p +  2\frac{\log p}{\log( \frac{n}{pq} )}\log 4 \right] \\
&  \leq \exp \left[ -\frac{1}{2} \log p \right] \\
& \to 0
\end{split} \end{equation}
as $n \to \infty$ since $\frac{2\log 4}{\log ( \frac{n}{pq})} \leq \frac{1}{2}$ eventually.

Finally, we check the convergence in probability of the error term $A = \frac{\widetilde{a_n}tr(X^TX)^{l+1}}{n^l}$. By Lemma \ref{L: cov2},
\begin{align*}
\mathbb{P} & \left[ A \geq \frac{\epsilon}{4} \right]  = \mathbb{P} \left[ \widetilde{a_n} \frac{tr(X^TX)^{l+1}}{n^l} \geq \frac{\epsilon}{4} \right] \\
& \leq \frac{16\widetilde{a_n}^2}{\epsilon^2n^{2l}}Var[tr(X^TX)^{l+1}] \\
& = \frac{16\widetilde{a_n}^2}{\epsilon^2n^{2l}} \left( 2(l+1)^2((p)_{l+1}(p)_lq + p(q)_l(q)_{l+1})  +Ce_{l+1,l+1} + Df_{l+1,l+1} \right) \\
& \leq \frac{\widetilde{C}p^{2l+1}q^24^{2l}}{n^{2l}}.
\end{align*}

As above, the choice of $l$ then guarantees that $\mathbb{P}[ A \geq \frac{\epsilon}{4}] \to 0$. 

\section{Combinatorics of Wishart Matrices}
\label{sec: Wishart}
The following result is a slight extension of a result in \cite{Bai} on means of traces of powers of Wishart matrices. The majority of the proof is the same as the one in \cite{Bai}. However, somewhat more careful estimates of the error are needed to complete the proof of the main theorem. 

\begin{lemma} \label{L: exp}
Let $ \{ p_n : n \geq 1 \}$ and $ \{q_n : n \geq 1 \}$ be two sequences of positive integers such that  $q_n \leq p_n$. For each n, let $X_n = (x_{ij})$ be a $p_n \times q_n$ matrix where $x_{ij}$ are independent standard Gaussian random variables. Then for each integer $h \geq 1$, 
\begin{equation*}
\mathbb{E}[tr(X_n^TX_n)^h]=\bigg( \sum_{r=0}^{h-1} (p_n)_{h-r}(q_n)_{r+1}\tfrac{1}{r+1}\tbinom{h}{r} \tbinom{h-1}{r} \bigg) \bigg( 1 + O\bigg( \tfrac{h}{p_n-h} \bigg) \bigg),
\end{equation*}
where $(p_n)_{h-r}$ and $(q_n)_{r+1}$ are falling factorials, which are defined by
\begin{equation*}
(x)_a = x (x-1)(x-2) \cdots (x-a+1).
\end{equation*}
\end{lemma}

\begin{proof}
 Write $(tr(X_n^TX_n)^h)$ as 
\begin{equation*} 
\sum_{1 \leq i_1,...,i_h \leq p} \sum_{1 \leq j_1,...,j_h \leq q} x_{i_1j_1}x_{i_1j_2}x_{i_2j_2} \cdots x_{i_{h-1}j_h}x_{i_hj_h} x_{i_hj_1}  = \sum_G x_G,
 \end{equation*}
where $G$ is a bipartite graph with the $i_k$ on a top line and the $j_k$ on a bottom line, with $h$ up-edges from $j_k$ to $i_k$ and $h$ down-edges from $i_k$ to $j_{k+1}$. We refer to such a graph as an S-graph. An edge $(i,j)$ in the S-graph corresponds to the variable $x_{ij}$. Now,
\begin{align*}
\mathbb{E}[tr(X_n^TX_n)^h] & =  \sum_{G} \mathbb{E}[X_G],
\end{align*}
where the sum is taken over all S-graphs $G$. If $G$ contains any edges of odd multiplicity then $\mathbb{E}[X_G]=0$ so that the proof reduces to the case where $G$ contains only edges of even multiplicity. Each S-graph $G$ contains $2h$ edges, hence at most $h$ distinct edges. It then follows that $G$ has at most $h+1$ distinct vertices. First consider the case when $G$ contains exactly $h$ distinct edges and thus $h+1$ distinct vertices. For each $r = 0, \cdots, h-1$ the calculation reduces to counting the number of graphs which have no single edges, $r+1$ non-coincident $j$-vertices and $h-r$ non-coincident $i$-vertices. 

Consider two S-graphs to be isomorphic if one can be converted to the other by permuting $\{ 1,...,p \}$ on the top line and $\{1,...,q \}$ on the bottom line. To count the number of isomorphism classes define $u_l = -1$ if the graph leaves a bottom vertex for the final time after the $l$th up edge and $u_l=0$ otherwise. Define $d_l = 1$ if the $l$th down edge leads to a new bottom vertex and $d_l = 0$ otherwise. The graph must return to the initial bottom vertex so $u_1 = 0$. Because the number of vertices seen for the final time cannot exceed the number of new vertices, we have $d_1 + \cdots + d_{l-1} + u_1 + \cdots + u_l \geq 0$ for every $l$.

There are $\binom{h}{r}$ ways to arrange $r$ ones into the $h$ positions of down edges that could lead to a new bottom vertex. There are $\binom{h-1}{r}$ ways to arrange $r$ minus ones into the $h-1$ positions of up edges that leave a bottom vertex for the last time. ($h-1$ since the first vertex can never be left for the last time.) Thus we have $\binom{h}{r} \binom{h-1}{r}$ ways to arrange our $d$-sequence and $u$-sequence. 

However, not all of these $\binom{h}{r} \binom{h-1}{r}$ graphs are a proper S-graph. An S-graph is improper if at at some point $d_1 + \cdots + d_{l-1} + u_1 + \cdots + u_l < 0$. To count the number of improper graphs, let $L$ be the first integer at which this happens. Then we must have $d_{L-1} = 0$ and $u_L = -1$. That is, we have just returned to a vertex we have seen before and left it for the last time. To fix it we must instead see a new vertex that we will return to again later. To do this, change $d_{L-1}$ to 1 and $u_L$ to 0. The initial bad sequences contained $r$ ones and $r$ minus ones. The fixed sequences now contain $r+1$ ones and $r-1$ minus ones. Therefore we have $\binom{h}{r+1} \binom{h-1}{r-1}$ bad sequences. Thus the number of isomorphism classes is $\binom{h}{r} \binom{h-1}{r} - \binom{h}{r+1} \binom{h-1}{r-1} = \frac{1}{r+1} \binom{h}{r} \binom{h-1}{r}$. The number of graphs in each isomorphism class is $p(p-1) \cdots (p-h+r+1)q(q-1)\cdots(q-r) = (p)_{h-r}(q)_{r+1}$. The number of S-graphs with exactly $h$ distinct edges (and therefore exactly $h+1$ distinct vertices) is $\sum_{r=0}^{h-1} (p)_{h-r}(q)_{r+1} \frac{1}{r+1} \binom{h}{r} \binom{h-1}{r}$. This is the main term in the expectation. We will next show that the sum of the remaining terms is of smaller order. 

Suppose that $G$ has $m<h$ distinct edges.  Let $r = 0, 1, \cdots, m-1$ and choose $r+1$ bottom vertices and $m-r$ top vertices. There are $\frac{1}{r+1} \binom{m}{r} \binom{m-1}{r}$ isomorphism classes and $(p)_{m-r}(q)_{r+1}$ graphs per class. Now, $G$ contains $m$ distinct edges. So there are $h-m$ (double) edges left to place within the graph. Each of these edges can overlap with any of the $m$ distinct edges already in the graph. So there are at most $m^{h-m}$ ways to arrange the edges of multiplicity more than two, and there are thus at most $m^{h-m} \sum_{r=0}^{m-1}(p)_{m-r}(q)_{r+1} \frac{1}{r+1} \binom{m}{r} \binom{m-1}{r}$ such S-graphs with $m$ distinct edges. 

Comparing the contribution of the $m$ edge case to that of the $h$ edge case,
\begin{align*}
& \frac{m^{h-m} \sum_{r=0}^{m-1}(p)_{m-r}(q)_{r+1} \frac{1}{r+1} \binom{m}{r} \binom{m-1}{r}}{\sum_{r=0}^{h-1} (p)_{h-r}(q)_{r+1} \frac{1}{r+1} \binom{h}{r} \binom{h-1}{r}} \\
& = \frac{m^{h-m} \sum_{r=0}^{m-1} (p)_{m-r}(q)_{r+1} \frac{1}{r+1} \binom{m}{r} \binom{m-1}{r} }{\sum_{r=0}^{h-1} (p-m+r)\cdots (p-h+r+1) (p)_{m-r}(q)_{r+1} \frac{1}{r+1} \binom{h}{r} \binom{h-1}{r}} \\
& \leq \frac{m^{h-m} \sum_{r=0}^{h-1} (p)_{m-r}(q)_{r+1} \frac{1}{r+1} \binom{h}{r} \binom{h-1}{r} }{(p-m)\cdots (p-h+1)\sum_{r=0}^{h-1} (p)_{m-r}(q)_{r+1} \frac{1}{r+1} \binom{h}{r} \binom{h-1}{r}} \\
& = \frac{m^{h-m}}{(p-m)(p-m-1)\cdots (p-h+1)} \\
& \leq \bigg( \frac{h}{p-h} \bigg)^{h-m}. \\
\end{align*}
Summing over all possibilities for $m$,
\begin{align*}
\sum_{m=1}^{h-1} \left( \frac{h}{p-h} \right)^{h-m} & = \left( \frac{h}{p-h} \right)^h  \sum_{m=1}^{h-1} \left( \frac{p-h}{h} \right)^m \\
& = \left( \frac{h}{p-h} \right)^h \frac{ \left( \frac{p-h}{h} \right)^h - 1}{\frac{p-h}{h} - 1} \\
& \leq \left( \frac{h}{p-h} \right)^h \frac{h \left( \frac{p-h}{h} \right)^h}{p - 2h} \\
& = \frac{h}{p-2h}.
\end{align*}
Then 
\begin{equation*}
\mathbb{E}[tr(X_n^TX_n)^h]=\bigg( \sum_{r=0}^{h-1} (p)_{h-r}(q)_{r+1}\frac{1}{r+1}\binom{h}{r} \binom{h-1}{r} \bigg) \bigg( 1 + O\bigg( \frac{h}{p-h} \bigg) \bigg).
\end{equation*} 
Note that the implicit constant does not depend on any of the parameters. 
\end{proof}
\vspace{1 cm}

Careful asymptotics for the covariances of traces of powers of Wishart matrices are given below. The proof uses both delicate combinatorial arguments and difficult estimates and has been broken up into two separate lemmas for clarity. The result is given first in Lemma \ref{L: cov2}, followed by Lemma \ref{L:  cov1} which gives the combinatorial part of the argument, then the proof of Lemma \ref{L: cov2} is completed using estimates.

Recall that the notation $(x)_a$ represents the  falling factorial, which is defined by
\begin{equation*}
(x)_a = x (x-1)(x-2) \cdots (x-a+1).
\end{equation*}

\begin{lemma} \label{L: cov2} Let $ \{ p_n : n \geq 1 \}$ and $ \{q_n : n \geq 1 \}$ be two sequences of positive integers such that $p_n \to \infty$ and $q_n \leq p_n$. For each n, let $X_n = (x_{ij})$ be a $p_n \times q_n$ matrix where $x_{ij}$ are independent standard Gaussian random variables. There exist constants $C, D>0$ and $e_{h,k}, f_{h,k}$ such that for integers $h \geq 1$ and $k \geq h$, 
\begin{align*}
Cov(tr(X_n^TX_n)^h,tr(X_n^TX_n)^k) & = 2hk((p)_{h}q(p)_{k-1}+p(q)_{h}(q)_{k-1})\\
& \qquad \qquad + Ce_{h,k} +Df_{h,k},
\end{align*}
where $e_{h,k} = o \left(kp^{h+k-2}q^24^{h+k} \right)$ and $f_{h,k} = o \left( k^8p^{h+k-2}q^24^{h+k} \right)$.
\end{lemma}

\begin{lemma} \label{L: cov1}
Let $ \{ p_n : n \geq 1 \}$ and $ \{q_n : n \geq 1 \}$ be two sequences of positive integers such that  $q_n \leq p_n$. For each n, let $X_n = (x_{ij})$ be a $p_n \times q_n$ matrix where $x_{ij}$ are independent standard Gaussian random variables. Then for integers $h \geq 1$ and $k \geq h$, 

\begin{equation}\begin{split}  \label{eq:1}
& Cov(tr(X_n^TX_n)^h,tr(X_n^TX_n)^k) \\
& = 2hk \left( \textstyle \sum_{r=0}^{h-1}  (p_n)_{h-r}(q_n)_{r+1} \tfrac{1}{r+1} \tbinom{h}{r} \tbinom{h-1}{r} \right) \\
& \qquad \times \left( \textstyle \sum_{s=0}^{k-1}(p_n)_{k-1-s}(q_n)_{s} \tfrac{1}{s+1} \tbinom{k}{s} \tbinom{k-1}{s} \right) \left( 1 + O \left( \tfrac{k^6}{p_n-h} \right) \right)\\
& + \left( 2 \left( k-h \right) (p_n)_h(q_n)_h \textstyle \sum_{r=0}^{k-h} (p_n)_{k-h-r}(q_n)_{r} \tfrac{1}{r+1} \tbinom{k-h+1}{r} \tbinom{k-h}{r} \right) \\
& \qquad \times \left( 1 + O \left( \tfrac{k^4}{p_n-k+h} \right) \right) \\
& + \bigg( \textstyle \sum_{l=2}^{h-1} 2 \left( h-l \right) \left( k-l \right) (p_n)_l(q_n)_l \\
&  \qquad \times \left( \textstyle \sum_{r=0}^{h-l} (p_n)_{h-l-r}(q_n)_{r} \tfrac{1}{r+1} \tbinom{h-l+1}{r} \tbinom{h-l}{r} \right) \\
&  \qquad \times \left(  \textstyle \sum_{s=0}^{k-l}(p_n)_{k-l-s}(q_n)_{s} \tfrac{1}{s+1} \tbinom{k-l+1}{s} \tbinom{k-l}{s} \right) \bigg) \left( 1 + O \left( \tfrac{k^{7}}{p_n-k} \right) \right).
\end{split} \end{equation}
\end{lemma}

\begin{proof}
Write $(tr(X_n^TX_n)^h)$ as 
\begin{equation*} 
\sum_{1 \leq i_1,...,i_h \leq p} \sum_{1 \leq j_1,...,j_h \leq q} x_{i_1j_1}x_{i_1j_2}x_{i_2j_2} \cdots x_{i_{h-1}j_h}x_{i_hj_h} x_{i_hj_1}  = \sum_G x_G,
 \end{equation*}
where $G$ is a bipartite graph with the $i_k$ on a top line and the $j_k$ on a bottom line, with $h$ up-edges from $j_k$ to $i_k$ and $h$ down-edges from $i_k$ to $j_{k+1}$. An edge $(i,j)$ in the S-graph corresponds to the variable $x_{ij}$. Now,
\begin{align*}
& Cov(tr(X^TX)^h, tr(X^TX)^k) \\
& = \mathbb{E}[tr(X^TX)^htr(X^TX)^k] - \mathbb{E}[tr(X^TX)^h]\mathbb{E}[tr(X^TX)^k] \\
& = \sum_{G,K}( \mathbb{E}[X_GX_K] - \mathbb{E}[X_G]\mathbb{E}[X_K] ),
\end{align*}
where $G,K$ are both S-graphs. If $G \cup K$ contains a single edge, or an edge of odd multiplicity, then either $G$ or $K$ also contains a single edge. In either case, $\mathbb{E}[X_{G}X_{K}]$ and $\mathbb{E}[X_{G}]\mathbb{E}[X_{K}]$ are both zero. If $G$ and $K$ do not have a coincident edge, then $\mathbb{E}[X_{G}X_{K}] = \mathbb{E}[X_{G}]\mathbb{E}[X_{K}]$ and the difference is zero. It thus suffices to consider the case where there are no edges of odd multiplicity and $G$ and $K$ have at least one edge in common. There are three main cases: 

(1) $\mathbb{E}[X_G] = \mathbb{E}[X_K] = 1$ and $G \cup K$ contains an edge of multiplicity four. There are 
\begin{equation*}\begin{split}
& \left( hk \sum_{r=0}^{h-1} \sum_{s=0}^{k-1} (p)_{h-r}(q)_{r+1} \frac{1}{r+1} \binom{h}{r} \binom{h-1}{r} (p)_{k-1-s}(q)_{s} \frac{1}{s+1} \binom{k}{s} \binom{k-1}{s} \right) \\
& \qquad \times \left( 1 + O \left( \frac{k^6}{p-h} \right) \right)
\end{split}\end{equation*}
such pairs $(G,K)$ of graphs. 

One can first build $G$ as in Lemma \ref{L: exp}. In this case, there are 
\begin{equation*} \sum_{r=0}^{h-1} (p)_{h-r}(q)_{r+1}\frac{1}{r+1}\binom{h}{r} \binom{h-1}{r}  \end{equation*} 
such graphs. 

Next, build the graph of $K$. There are $h$ possible choices of edges in the graph of $G$ with which one of the edges of $K$ can coincide and $k$ possible times in the construction of $K$ at which a coincident edge may be added. $K$ has an edge in common with $G$ and therefore $K$ can have at most $k-1$ new vertices. Let $s = 0, \cdots, k-1$ and choose $s$ bottom vertices and $k-s-1$ top vertices. As with $G$, the number of isomorphism classes is $\binom{k}{s} \binom{k-1}{s} - \binom{k}{s+1} \binom{k-1}{s-1} = \frac{1}{s+1} \binom{k}{s} \binom{k-1}{s}$. There will be  $ (p)_{k-1-s}(q)_s$ graphs in each isomorphism class. 

Thus there are 
\begin{equation*} hk \sum_{r=0}^{h-1} \sum_{s=0}^{k-1} (p)_{h-r}(q)_{r+1} \frac{1}{r+1} \binom{h}{r} \binom{h-1}{r} (p)_{k-1-s}(q)_{s} \frac{1}{s+1} \binom{k}{s} \binom{k-1}{s} \end{equation*}
 possible graphs with exactly $h+1$ distinct vertices in $G$ and $k-1$ additional distinct vertices in $K$. $G$ and $K$ both contain only edges of multiplicity two. Thus $\mathbb{E}[X_GX_K] = 3$ and $\mathbb{E}[X_G] = \mathbb{E}[X_K] = 1$. This gives the factor of 2 in the first term of the expression in the lemma. 

Any graphs with fewer distinct edges will produce terms of smaller order than the term found above. The error is computed by counting the number of possible graphs. Therefore, suppose now that there is at least one fewer distinct edge in $G \cup K$, so that either $G$ contains at most $h-1$ distinct edges or $K$ contains at most $k-2$ distinct vertices. Without loss of generality, we assume that it is $G$ that has a reduced number of edges, so that $G$ has at most $h-1$ distinct edges and $K$ has at most $k-1$ distinct vertices. Let $G$ have $m < h$ distinct edges, hence $m+1$ non-coincident vertices, and $K$ have $n \leq k-1$ non-coincident vertices. Taking $r = 0, \cdots, m-1$, $r+1$ bottom vertices, and $m-r$ top vertices, there are $\frac{1}{r+1} \binom{m}{r} \binom{m-1}{r}$ isomorphism classes and $(p)_{m-r}(q)_{r+1}$ graphs per class. Now, $G$ contains $m$ distinct edges. So there are $h-m$ (double) edges left to place within the graph. Each of these (double) edges can coincide with any of the $m$ distinct edges. So there are at most $m^{h-m}$ ways to arrange the edges of multiplicity greater than two. In the same way, there are at most $n^{k-n}$ ways to arrange the edges of multiplicity greater than two in the construction of $K$. Let $s = 0, \cdots, n$, and choose $s$ bottom vertices and $n-s$ top vertices in the graph of $K$, then there are $\frac{1}{s+1} \binom{n+1}{s} \binom{n}{s}$ isomorphism classes and $(p)_{n-s}(q)_s$ graphs per class. ($n+1$ for the $n$ new vertices in $K$ and the one overlapping vertex). The number of such graphs is at most
\begin{equation*}
m^{h-m}n^{k-n} \textstyle \sum_{r=0}^{m-1} (p)_{m-r}(q)_{r+1} \tfrac{1}{r+1} \tbinom{m}{r} \tbinom{m-1}{r} \textstyle \sum_{s=0}^n (p)_{n-s}(q)_s \tfrac{1}{s+1} \tbinom{n+1}{s} \tbinom{n}{s}.
\end{equation*}
Lastly, observe that $\mathbb{E}[X_GX_K]-\mathbb{E}[X_G]\mathbb{E}[X_K] \leq \mathbb{E}[X_GX_K]$, and this expected value will be largest when all but one of the edges in $G \cup K$ has multiplicity two. That is, when $G$ has an edge of multiplicity $2(h-m+1)$ that overlaps with an edge of $K$ of multiplicity $2(k-n+1)$, this one edge will have multiplicity $2h + 2k -2m - 2n +4$. Thus, by a quantitative version of Stirling's formula (Lemma \ref{L: stirling}),
\begin{align*}
E & := \mathbb{E}[X_GX_K]  = (2h + 2k - 2m - 2n + 3)!! \\
& = \frac{(2h + 2k - 2m -2n + 4)!}{2^{h+k-m-n+2}(h+k-m-n+2)!} \\
& \leq \frac{e}{\sqrt{\pi}} \bigg( \frac{2h+2k-2m-2n+4}{e} \bigg)^{h+k-m-n+2}.
\end{align*}

The ratio of this term to the contribution from those graphs with exactly $h+1$ and $k-1$ distinct vertices that was calculated earlier is

\begin{align*}
& \tfrac{Em^{h-m}n^{k-n}\sum\limits_{r=0}^{m-1} (p)_{m-r}(q)_{r+1} \tfrac{1}{r+1} \tbinom{m}{r} \tbinom{m-1}{r} \sum\limits_{s=0}^n (p)_{n-s}(q)_s \tfrac{1}{s+1} \tbinom{n+1}{s} \tbinom{n}{s}}{2hk \sum\limits_{r=0}^{h-1} \sum\limits_{s=0}^{k-1} (p)_{h-r}(q)_{r+1} \tfrac{1}{r+1} \tbinom{h}{r} \tbinom{h-1}{r} (p)_{k-1-s}(q)_{s} \tfrac{1}{s+1} \tbinom{k}{s} \tbinom{k-1}{s}}\\
& \leq \tfrac{Em^{h-m}n^{k-n} \sum\limits_{r=0}^{m-1} (p)_{m-r}(q)_{r+1} \tfrac{1}{r+1} \tbinom{m}{r} \tbinom{m-1}{r} \sum\limits_{s=0}^n (p)_{n-s}(q)_s \tfrac{1}{s+1} \tbinom{n+1}{s} \tbinom{n}{s}}{2hk (p-h)^{h-m}(p-k)^{k-n-1}\sum\limits_{r=0}^{h-1} \sum\limits_{s=0}^{k-1} (p)_{m-r}(q)_{r+1} \tfrac{1}{r+1} \tbinom{h}{r} \tbinom{h-1}{r} (p)_{n-s}(q)_{s} \tfrac{1}{s+1} \tbinom{k}{s} \tbinom{k-1}{s}}\\
& \leq \tfrac{Em^{h-m}n^{k-n}\sum\limits_{r=0}^{h-1} (p)_{m-r}(q)_{r+1} \tfrac{1}{r+1} \tbinom{h}{r} \tbinom{h-1}{r} \sum\limits_{s=0}^{k-1} (p)_{n-s}(q)_s \tfrac{1}{s+1} \tbinom{k}{s} \tbinom{k-1}{s}}{2hk (p-h)^{h-m}(p-k)^{k-n-1}\sum\limits_{r=0}^{h-1} \sum\limits_{s=0}^{k-1} (p)_{m-r}(q)_{r+1} \tfrac{1}{r+1} \tbinom{h}{r} \tbinom{h-1}{r} (p)_{n-s}(q)_{s} \tfrac{1}{s+1} \tbinom{k}{s} \tbinom{k-1}{s}}\\
& = \frac{Em^{h-m}n^{k-n}}{2hk(p-h)^{h-m}(p-k)^{k-n-1}}\\
& \leq \frac{\frac{e}{\sqrt{\pi}} \bigg( \frac{2h+2k-2m-2n+4}{e} \bigg)^{h+k-m-n+2}h^{h-m-1}k^{k-n-1}}{2(p-h)^{h-m}(p-k)^{k-n-1}} \\
& \leq \frac{e}{2\sqrt{\pi}} \bigg( \frac{(2h+2k)h}{e(p-h)} \bigg)^{h-m-1} \bigg( \frac{(2h+2k)k}{e(p-k)} \bigg)^{k-n-1} \bigg( \frac{ (2h+2k)^4}{e^4(p-h)} \bigg). \\
\end{align*}

Summing over all possible combinations of $m$ and $n$ gives the contribution from all the graphs that can be constructed under the assumption of one or more additional coincident edges as follows, 

\begin{align*}
& \sum_{m=1}^{h-1} \sum_{n=1}^{k-1} \frac{e}{2\sqrt{\pi}} \bigg( \frac{(2h+2k)h}{e(p-h)} \bigg)^{h-m-1} \bigg( \frac{(2h+2k)k}{e(p-k)} \bigg)^{k-n-1} \bigg( \frac{ (2h+2k)^4}{e^4(p-h)} \bigg) \\
& = \frac{e}{2\sqrt{\pi}}  \bigg( \frac{ (2h+2k)^4}{e^4(p-h)} \bigg) \left( \frac{h(2h-2k)}{e(p-h)} \right)^{h-1} \left( \frac{k(2h+2k)}{e(p-k)} \right)^{k-1} \\
& \qquad \times \sum_{m=1}^{h-1} \left( \frac{e(p-h)}{h(2h+2k)} \right)^m \sum_{n=1}^{k-1} \left( \frac{e(p-k)}{k(2h+2k)} \right)^n \\
& \leq \frac{e}{2\sqrt{\pi}} hk  \bigg( \frac{ (2h+2k)^4}{e^4(p-h)} \bigg) \left( \frac{h(2h-2k)}{e(p-h)} \right)^{h-1} \left( \frac{k(2h+2k)}{e(p-k)} \right)^{k-1} \\
& \qquad \times \left( \frac{e(p-h)}{h(2h+2k)} \right)^{h-1} \left( \frac{e(p-k)}{k(2h+2k)} \right)^{k-1} \\
& = \frac{e}{2\sqrt{\pi}}  \frac{hk(2h+2k)^4}{e^4(p-h)} \\
& \leq \frac{1}{2\sqrt{\pi}}  \frac{4^4k^6}{e^3(p-h)}.
\end{align*}

(2) Assume that $\mathbb{E}[X_G] = \mathbb{E}[X_K] = 0$, $\mathbb{E}[X_GX_K]$ does not equal $0$, and $2<h<k$ where $G$ is a graph consisting of $2h$ single edges. Then $K$ must contain a subgraph that exactly overlaps the single edges of $G$ in order that $\mathbb{E}[X_GX_K] \neq 0$. There are $2(p)_h(q)_h$ ways to construct $G$ since there are two possible orientations and $(p)_h(q)_h$ labels. Now $K$  will have $2k - 2h$  vertices not contained in the subgraph of single edges, hence at most $k-h$ distinct edges, and at most $k-h$ distinct vertices. First consider the case with exactly $k-h$ distinct vertices. As in the previous case, this will be the dominating term for case (2). The graph of $G$ can be attached onto $K$ in $k-h$ different places. Let $r= 0,...,k-h$. We again consider all possible ways of breaking the $k-h$ vertices into top and bottom vertices by choosing $r$ bottom vertices and $k-h-r$ top vertices. 

There are $\binom{k-h}{r}$ ways to arrange $r$ ones into the $k-h$ positions of down edges that could lead to new bottom vertices. Since the initial bottom vertex may be in the subgraph of single edges, it is possible to leave every bottom vertex for the last time, so there are $\binom{k-h}{r}$ ways to arrange $r$ minus ones into the $k-h$ possible positions of up edges that leave a bottom vertex for the last time. As before, not all of these $\binom{k-h}{r} \binom{k-h}{r}$ arrangements are proper S-graphs. There are $\binom{k-h}{r+1} \binom{k-h}{r-1}$ bad sequences.  Then there are $\binom{k-h}{r} \binom{k-h}{r} - \binom{k-h}{r+1} \binom{k-h}{r-1} =\frac{1}{r+1}\binom{k-h+1}{r} \binom{k-h}{r}$ possible isomorphism classes. Thus when $k > h$, there are 
\begin{equation*}
2(k-h)(p)_h(q)_h \sum_{r=0}^{k-h} (p)_{k-h-r}(q)_{r} \frac{1}{r+1} \binom{k-h+1}{r} \binom{k-h}{r}
\end{equation*}
graphs. 

Now, to compute the error, suppose there is at least one less distinct edge in $K$. Assume there are $n < k-h$ distinct vertices. As before, the graph of $G$ can be attached onto $K$ in $n$ places. There are $k-h-n$ double edges left to place, each of which can coincide with any of the $n$ distinct edges. So there are $n^{k-h-n}$ ways to arrange the non-distinct edges. Take $r$ bottom vertices and $n-r$ top vertices where $r= 0, \cdots, n$. Then there are $\binom{n+1}{r} \binom{n}{r}$ possible isomorphism classes of graphs and $(p)_{n-r}(q)_r$ graphs per class. Such graphs will have maximal expectation when all of the edges have multiplicity two, except for one edge of multiplicity $2(k-h-n+1)$. By Stirling's formula, $\mathbb{E} \leq \frac{e}{\sqrt{\pi}}  \bigg( \frac{2k-2h-2n+2}{e} \bigg)^{k-h-n+1}$. Comparing to the case with exactly $k-h$ distinct vertices, 

\begin{align*}
& \frac{\frac{e}{\sqrt{\pi}} \bigg( \frac{2k-2h-2n+2}{e} \bigg)^{k-h-n+1} n^{k-h-n+1} 2(p)_h(q)_h \sum_{r=0}^n (p)_{n-r}(q)_r \frac{1}{r+1} \binom{n+1}{r} \binom{n}{r}} {2(k-h)(p)_h(q)_h \sum_{r=0}^{k-h} (p)_{k-h-r}(q)_{r} \frac{1}{r+1} \binom{k-h+1}{r} \binom{k-h}{r}} \\
& \leq \frac{ \frac{e}{\sqrt{\pi}} \bigg( \frac{2k-2h-2n+2}{e} \bigg)^{k-h-n+1} n^{k-h-n+1}}{(k-h)(p-k+h)^{k-h-n}} \\
& \leq \frac{ \frac{e}{\sqrt{\pi}} \bigg( \frac{2k-2h-2n+2}{e} \bigg)^{k-h-n+1} (k-h)^{k-h-n}}{(p-k+h)^{k-h-n}} \\
& = \frac{e}{\sqrt{\pi}} \bigg( \frac{(k-h)(2k-2h-2n+2)}{e(p-k+h)} \bigg)^{k-h-n} \left( \frac{2k-2h-2n+2}{e} \right) \\
& \leq \frac{e}{\sqrt{\pi}} \bigg( \frac{2k^2}{e(p-k+h)} \bigg)^{k-h-n} \left( \frac{2k}{e} \right).
\end{align*}
 
Summing over all possible constructions,
 \begin{align*}
  \sum_{n=0}^{k-h-1} & \frac{e}{\sqrt{\pi}}  \bigg( \frac{2k^2}{e(p-k+h)} \bigg)^{k-h-n} \frac{2k}{e} \\ & = \frac{2k}{\sqrt{\pi}}   \bigg( \frac{2k^2}{e(p-k+h)} \bigg)^{k-h} \sum_{n=0}^{k-h-1} \bigg( \frac{e(p-k+h)}{2k} \bigg)^n \\
 & \leq \frac{2k^2}{\sqrt{\pi}}   \bigg( \frac{2k^2}{e(p-k+h)} \bigg)^{k-h} \bigg( \frac{e(p-k+h)}{2k} \bigg)^{k-h-1} \\
 & = \frac{4k^4}{\sqrt{\pi}e(p-k+h)}. \\
 \end{align*}
Then there are 
 \begin{equation*}
 \left( 2 \left( k-h \right )(p)_h(q)_h \sum\limits_{r=0}^{k-h} (p)_{k-h-r}(q)_{r} \tfrac{1}{r+1} \tbinom{k-h+1}{r} \tbinom{k-h}{r} \right) \left( 1 + O \left( \tfrac{k^4}{p-k+h} \right) \right)
 \end{equation*}
 possible graphs from case (2). 

(3) Assume that $\mathbb{E}[X_G] = \mathbb{E}[X_K] = 0$, $\mathbb{E}[X_GX_K]$ does not equal $0$, and $G$ is not a graph consisting of $2h$ single edges. Then for each $2 \leq l \leq h-1$, we will count the number of ways that both $G$ and $K$ could contain proper subgraphs consisting of $2l$ single edges that overlap exactly. In this case, for each $l$, there are
\begin{equation*} \begin{split}
& 2(h-l)(k-l) (p)_l(q)_l \\
& \quad \times \textstyle \sum_{r=0}^{h-l} (p)_{h-l-r}(q)_{r} \frac{1}{r+1} \binom{h-l+1}{r} \binom{h-l}{r} \textstyle \sum_{s=0}^{k-l} (p)_{k-l-s}(q)_{s} \frac{1}{s+1} \binom{k-l+1}{s} \binom{k-l}{s}
\end{split}\end{equation*}
possible graphs. 

This can be seen by first building up the closed cycle with $2l$ vertices. There are two orientations for the cycle and $(p)_l(q)_l$ choices of vertices. The rest of $G$ will have at most $h-l$ remaining distinct edges not within this cycle. The rest of $K$ will have at most $k-l$ distinct edges. First, consider the case when $G$ has exactly $h-l$ distinct edges and $K$ has exactly $k-l$ distinct edges. This will be the leading term for case (3). The subgraph of $2l$ single edges can be inserted into the construction of $G$ before any one of these edges. $G$ has $2h-2l$ vertices not in the cycle. Each edge must have multiplicity two and the cycle must attach to $G$ at one vertex. So $G$ has exactly $h-l$ distinct vertices outside of the cycle. Let $r = 0, \cdots, h-l$ and take $r$ bottom vertices and $h-l-r$ top vertices. As before, there are $\frac{1}{r+1} \binom{h-l+1}{r} \binom{h-l}{r}$ isomorphism classes and $(p)_{h-l-r}(q)_r$ choices of labels. Then build up $K$ in the same way by choosing $s = 0, \cdots, k-l$ bottom vertices and $k-l-s$ top vertices and inserting the cycle after any of the $k-l$ edges. Then there are $(k-l)\sum_{s=0}^{k-l} \frac{1}{s+1}\binom{k-l+1}{s} \binom{k-l}{s} (p)_{k-l-s}(q)_s$ non-isomorphic graphs. 

To compute the error, assume for some $l$ that $G \cup K$ contains at least one less distinct edge. First, we will assume that all of the edges with multiplicity more than two all lie within the cycle consisting of $2l$ distinct edges. The cycle part of $G$ will then contain $2l$ distinct edges. Let the $i^{th}$ edge have multiplicity $m_i$. Then $m_i \geq 1$, each $m_i$ is odd by assumption, and $2l \leq \sum_{i=1}^{2l} m_i := m \leq 2h$. Similarly, the cycle part of $K$ consists of the overlapping $2l$ edges each with multiplicity $n_i \geq 1$, odd, and $2l < \sum_{i=1}^{2l} n_i := n \leq 2k$. Where, without loss of generality, it has been assumed that at least one edge in the cycle piece of $K$ has multiplicity at least 3. There are two orientations for the subgraph and $(p)_l(q)_l$ choices of vertices. 

Next, build the rest of $G$. There are $\alpha := \frac{2h-m}{2}$ (double) edges left to place outside of the cycle. Note that $0 \leq \alpha \leq h-l$. Let $r = 0, \cdots, \alpha$ and choose $r$ bottom vertices and $\alpha -r$ top vertices. There are $\frac{1}{r+1} \binom{\alpha+1}{r} \binom{\alpha}{r}$ ways to arrange the edges in $G$ outside of the cycle and $(p)_{\alpha -r}(q)_r$ ways to label the edges. Furthermore, there are $\alpha$ places in the construction of $G$ where the cycle can be inserted. Similarly, build up the rest of $K$ by taking $\beta := \frac{2k-n}{2}$ (double) edges outside of the cycle with $0 \leq \beta < h-l$.  

The largest expectation occurs when all of the edges in the combined cycles are double edges except for one with multiplicity $m+n - (4l-2)$. The corresponding expectation for such a graph is 
\begin{equation*}\begin{split}
E  & := (m+n-4l+1)!! \\
& = \frac{(m+n-4l+2)!}{2^{\frac{m}{2} + \frac{n}{2} - 2l+1}(\frac{m}{2}+\frac{n}{2}-2l+1)!}\\
& \leq \frac{e}{\sqrt{\pi}} \bigg( \frac{m+n-4l+2}{e} \bigg)^{\frac{m}{2} + \frac{n}{2} - 2l +1},
\end{split} \end{equation*}
 by Stirling's Formula. Comparing this to the term for the same cycle of $2l$ edges but where each edge has exact multiplicity two, 

\begin{align*}
& \tfrac{2E(p)_l(q)_l \alpha \beta \sum_{r=0}^{\alpha} (p)_{\alpha -r}(q)_r \frac{1}{r+1} \binom{\alpha+1}{r} \binom{\alpha}{r} \sum_{s=0}^{\beta} (p)_{\beta -s}(q)_s \frac{1}{s+1} \binom{\beta+1}{s} \binom{\beta}{s}} {2(h-l)(k-l)(p)_l(q)_l \sum_{r=0}^{h-l} \sum_{s=0}^{k-l} (p)_{h-l-r}(q)_r \frac{1}{r+1} \binom{h-l+1}{r} \binom{h-l}{r} (p)_{k-l-s}(q)_s \frac{1}{s+1} \binom{k-l+1}{s} \binom{k-l}{s}} \\
& \leq \frac{ E} { (p-h+l)^{h-l-\alpha} (p-k+l)^{k-l+\beta} }.
\end{align*}

Now, $\frac{m}{2} = h- \alpha$ and $\frac{n}{2} = k - \beta$. So $\frac{m}{2} + \frac{n}{2} - 2l + 1 = h +k - 2l - \alpha - \beta +1$. Therefore, the ratio of the terms is 
\begin{align*}
\frac{e}{\sqrt{\pi}} & \bigg( \tfrac{m+n-4l+2}{e(p-h+l)} \bigg)^{h-l - \alpha}\bigg( \tfrac{m+n-4l+2}{e(p-k+l)} \bigg)^{k-l - \beta} \bigg( \tfrac{m+n-4l+2}{e} \bigg) \\
& \leq  \bigg( \tfrac{2h+2k-4l+2}{e(p-h+l)} \bigg)^{h-l - \alpha}\bigg( \tfrac{2h+2k-4l+2}{e(p-k+l)} \bigg)^{k-l - \beta} \bigg( \tfrac{2h+2k-4l+2}{\sqrt{\pi}} \bigg).
 \end{align*}
Summing over all possible multiplicities within the cycles, 
\begin{align*}
 \sum_{\alpha = 0}^{h-l} \sum_{\beta = 0}^{k-l-1} & \left( \tfrac{2h+2k-4l+2}{e(p-h+l)} \right)^{h-l - \alpha}\left( \tfrac{2h+2k-4l+2}{e(p-k+l)} \right)^{k-l - \beta} \left( \tfrac{2h+2k-4l+2}{\sqrt{\pi}} \right)\\
& \leq (h-l)(k-l) \frac{(2h+2k-4l+2)^2}{e\sqrt{\pi}(p-k+l)} \\
& \leq \frac{16k^4}{e\sqrt{\pi}(p-k+l)}.
\end{align*}

Further summing over all possible numbers of edges $l$ within the cycle, 
\begin{align*}
& \sum_{l=2}^{h-1} \frac{16k^4}{e\sqrt{\pi}(p-k+l)} \leq \frac{(h-1)16k^4}{e\sqrt{\pi}(p-k)} \leq \frac{16k^5}{e\sqrt{\pi}(p-k)}.
\end{align*}

Finally, extend to the case where some of the edges of multiplicity greater than two could lay outside of the cycle. Fix the arrangement of edges inside the cycle and consider the possibility that there is at least one less distinct edge outside of the cycle in $G$. Take $a < \alpha$ distinct edges in $G$ outside of the cycle and $b \leq \beta$ distinct edges in $K$. Then there are $\alpha - a $ double edges left to place in the graph of $G$ and $a^{\alpha - a}$ places to put them. There are also $b^{\beta - b}$ ways to arrange the remaining double edges in $K$. There are $ab$ places to insert the cycle into both graphs and $\sum_{r=0}^a (p)_{a-r}(q)_r \frac{1}{r+1} \binom{a+1}{r} \binom{a}{r} \sum_{s=0}^b (p)_{b-s}(q)_s \frac{1}{s+1} \binom{b+1}{s} \binom{b}{s}$ ways to arrange and label all of the distinct edges outside of the cycle. The maximum expectation in this case occurs when not only are all but one of the edges in the cycle double edges, but also all but one of the edges outside of the cycle are double edges. Then the two remaining edges have multiplicity  $m+n - (4l-2)$, as before, and $2(\alpha - a+1) + 2(\beta - b +1)$. An application of Stirling's Formula then gives a maximum expectation of $E_{\alpha \beta}$, where $E_{\alpha \beta}$ is bounded by $\frac{e}{\sqrt{\pi}} \bigg( \frac{2\alpha + 2 \beta - 2a - 2b +4}{e} \bigg)^{\alpha + \beta - a-b+2}\frac{e}{\sqrt{\pi}} \bigg( \frac{m+n-4l+2}{e} \bigg)^{\frac{m}{2} + \frac{n}{2} - 2l +1}$. Comparing to the previous case, 

\begin{align*}
& \tfrac{E_{\alpha \beta} a^{\alpha -a +1}b^{\beta - b+1} 2(p)_l(q)_l \sum\limits_{r=0}^a (p)_{a-r}(q)_r \frac{1}{r+1} \binom{a+1}{r} \binom{a}{r} \sum\limits_{s=0}^b (p)_{b-s}(q)_s \frac{1}{s+1} \binom{b+1}{s} \binom{b}{s}  }{\frac{e}{\sqrt{\pi}} ( \frac{m+n-4l+2}{e} )^{\frac{m}{2}+\frac{n}{2}-2l+1} 2(p)_l(q)_l \alpha \beta \sum\limits_{r=0}^{\alpha} (p)_{\alpha -r}(q)_r \frac{1}{r+1} \binom{\alpha+1}{r} \binom{\alpha}{r} \sum\limits_{s=0}^{\beta} (p)_{\beta -s}(q)_s \frac{1}{s+1} \binom{\beta+1}{s} \binom{\beta}{s}} \\
& \leq \frac{\frac{e}{\sqrt{\pi}} \bigg( \frac{2\alpha + 2 \beta - 2a - 2b +4}{e} \bigg)^{\alpha + \beta - a-b+2} (h-k)^{\alpha + \beta - a- b}} {(p-\alpha)^{\alpha -a}(p-\beta)^{\beta - b}} \\
& \leq \frac{e}{\sqrt{\pi}}\bigg( \frac{4k^2}{e(p-\alpha)} \bigg)^{\alpha -a}\bigg( \frac{4k^2}{e(p-\beta)} \bigg)^{\beta -b} \bigg( \frac{4k^2}{e} \bigg)^2.
\end{align*}

Summing over all possible $a$, $b$ and $l$,
\begin{align*}
 \sum_{l=2}^{h-1} \sum_{a=0}^{\alpha-1} & \sum_{b=0}^{\beta} \frac{e}{\sqrt{\pi}}\bigg( \frac{4k^2}{e(p-\alpha)} \bigg)^{\alpha -a}\bigg( \frac{4k^2}{e(p-\beta)} \bigg)^{\beta -b} \bigg( \frac{4k^2}{e} \bigg)^2 \\
 &  \leq \sum_{l=2}^{h-1} \frac{64k^6}{e^2\sqrt{\pi}(p-\beta)} \\
& \leq \frac{64k^7}{e^2\sqrt{\pi}(p-k)}.
\end{align*}
This completes the proof.
\end{proof}

\begin{proof}[Proof of Lemma \ref{L: cov2}] 
The expression for the covariance may be simplified by finding bounds on each of the three terms in equation \eqref{eq:1}. Beginning with the first term, 
\begin{align*}
& 2hk \left( \textstyle \sum_{r=0}^{h-1} (p)_{h-r}(q)_{r+1}\tfrac{1}{r+1} \tbinom{h}{r}\tbinom{h-1}{r} \right) \left( \textstyle \sum_{s=0}^{k-1} (p)_{k-s-1}(q)_{s} \tfrac{1}{s+1} \tbinom{k}{s} \tbinom{k-1}{s} \right) \\
& = 2hk \left( (p)_hq+ S_1 + p(q)_h \right) \left( (p)_{k-1} + S_2 + (q)_{k-1} \right), 
\end{align*} 
where \begin{equation*} S_1 = \sum_{r=1}^{h-2} (p)_{h-r}(q)_{r+1}\frac{1}{r+1} \binom{h}{r}\binom{h-1}{r}\end{equation*} and \begin{equation*}S_2 = \sum_{s=1}^{k-2} (p)_{k-s-1}(q)_{s} \frac{1}{s+1} \binom{k}{s} \binom{k-1}{s}.\end{equation*}
If $h=1$, then $S_1 = 0$. If $h \geq 2$, since $\frac{q}{p} \leq 1$, 
\begin{equation} \begin{split} \label{eq:3}
S_1 & = \sum_{r=1}^{h-2} (p)_{h-r}(q)_{r+1}\frac{1}{r+1} \binom{h}{r}\binom{h-1}{r}\\
& \leq \sum_{r=1}^{h-2} p^{h-r}q^{r+1} \frac{1}{r+1} \binom{h}{r} \binom{h-1}{r} \\
& = p^{h-1}q^2 \sum_{r=1}^{h-2} \bigg( \frac{q}{p} \bigg)^{r-1}  \frac{h!}{r!(h-r)!} \frac{(h-1)!}{(r+1)!(h-1-r)!} \\
& \leq p^{h-1}q^{2} \frac{1}{h} \sum_{r=1}^{h-2} \binom{h}{h-r} \binom{h}{r+1} \\
& \leq p^{h-1}q^{2} \frac{1}{h} \binom{2h}{h+1} \\
& \leq \frac{p^{h-1}q^{2}}{h} \frac{ e (2h)^{2h+1/2}e^{-2h}}{2 \pi (h+1)^{h+3/2}e^{-h-1}(h-1)^{h-1/2}e^{-h+1}} \\
& \leq  \frac{2^{2h}p^{h-1}q^{2}}{h^{3/2}} ,
\end{split} \end{equation}
where a quantitative form of Stirling's approximation (Lemma \ref{L: stirling}) was used in the second to last line. That is,
\begin{equation*}
0 \leq S_1 \leq \frac{2^{2h}p^{h-1}q^{2}}{h^{3/2}};
\end{equation*}
similarly, 
\begin{align*}
0 \leq S_2  \leq \frac{2^{2k} p^{k-2}q}{k^{3/2}}.
\end{align*} 
Then the first term of the covariance in equation \eqref{eq:1} is
\begin{align*}
2hk & \big(  (p)_{h}q(p)_{k-1} + (p)_hq S_2 + (p)_h(q)_k +(p)_{k-1}S_1 + S_1S_2 \\
&  \qquad \quad + (q)_{k-1}S_1 + (p)_k(q)_h + p(q)_hS_2 + p(q)_{h}(q)_{k-1} \big) \left( 1 + O \left( \tfrac{k^6}{p-h} \right) \right)\\
& = 2hk((p)_{h}q(p)_{k-1}+p(q)_{h}(q)_{k-1}) + \delta_{h,k},
\end{align*} 
where 
\begin{equation*} \begin{split}
& \delta_{h,k} \leq  2hk \bigg( \frac{4p^{h+k-2}q^{2}2^{2k}}{k^{3/2}} + \frac{2p^{h+k-2}q^{2}2^{2h}}{h^{3/2}} + \frac{p^{h+k-3}q^{3}2^{2h+2k}}{(k)^{3/2}} \bigg) \\
& + O \left( \frac{k^7h p^{h+k-1}q}{p-h} + \frac{k^7hpq^{h+k-1}}{p-h} + \frac{k^7h2^{2h+2k}p^{h+k-2}q^2}{\sqrt{h}(p-h)} \right)  \\
& \leq C\frac{p^{h+k-2}q^{2}2^{2h+2k}k}{\sqrt{h}} + \widetilde{C}2^{2h+2k}\frac{hk^7p^{h+k-1}q}{p-h}.
\end{split} \end{equation*}
When $h=k$ the second term in the covariance formula \eqref{eq:1} is zero. When $h<k$ the second sum in the covariance formula is bounded as follows,
\begin{equation*} \begin{split}
2(k-h)(p)_h & (q)_h  \sum_{r=0}^{k-h} (p)_{k-h-r}(q)_r \frac{1}{r+1} \binom{k-h+1}{r} \binom{k-h}{r} \\
& \leq 2(k-h)p^hq^h \sum_{r=0}^{k-h} p^{k-h-r}q^r \frac{1}{r+1} \binom{k-h+1}{r} \binom{k-h}{r} \\
& = 2(k-h)p^kq^h \sum_{r=0}^{k-h} \left( \frac{q}{p} \right)^r \frac{(k-h+1)!}{r!(k-h+1-r)!} \frac{(k-h)!}{(r+1)!(k-h-r)!} \\
& \leq 2p^kq^h\frac{k-h}{k-h+1} \sum_{r=0}^{k-h} \binom{k-h+1}{k-h-r} \binom{k-h+1}{r}\\
& \leq 2p^kq^h \binom{2(k-h+1)}{k-h} \\
& \leq \frac{e^2p^kq^h2^{2(k-h)+5/2}}{\pi}.
\end{split}\end{equation*}
Therefore, we have
\begin{equation*} \begin{split}
\bigg( & 2(k-h)(p)_h(q)_h \sum_{r=0}^{k-h} (p)_{k-h-r}(q)_r \tfrac{1}{r+1} \tbinom{k-h+1}{r} \tbinom{k-h}{r} \bigg) \left( 1 + O \left( \tfrac{k^4}{p-k+h} \right) \right) \\
& \leq Cp^kq^h2^{2(k-h)} + \widetilde{C}\frac{2^{2k-2h}k^4p^{h+k-2}q^2}{p-k+h}. 
\end{split}\end{equation*}

For the last term in the covariance formula \eqref{eq:1} the sums involved are bounded similarly to those in the second term:
\begin{equation*} \begin{split}
\sum_{r=0}^{h-l} (p)_{h-l-r}(q)_r \frac{1}{r+1} \binom{h-l+1}{r} \binom{h-l}{r} \leq \frac{e^2p^{h-l}2^{2h-2l+3/2}}{\pi (h-l+1)}
\end{split} \end{equation*}
and
\begin{equation*} \begin{split}
\sum_{s=0}^{k-l} (p)_{k-l-s}(q)_s \frac{1}{s+1} \binom{k-l+1}{s} \binom{k-l}{s} \leq \frac{e^2p^{k-l}2^{2k-2l+3/2}}{\pi (k-l+1)}.
\end{split} \end{equation*}
The last term is then bounded as follows:
\begin{equation*} \begin{split}
& \sum_{l=2}^{h-1} 2 \left( h-l \right) \left( k-l \right) (p)_l(q)_l \\
& \qquad \times \sum_{r=0}^{h-l} (p)_{h-l-r}q^{r} \tfrac{1}{r+1} \tbinom{h-l}{r} \tbinom{h-l-1}{r} \sum_{s=0}^{k-l} (p)_{k-l-s}q^{s} \tfrac{1}{s+1} \tbinom{k-l}{s} \tbinom{k-l-1}{s} \\
& \leq \frac{2^4e^4}{\pi^2} \sum_{l=2}^{h-1}(h-l)(k-l)p^lq^l\frac{p^{h+k-2l}2^{2h+2k-4l}}{(h-l+1)(k-l+1)}  \\
& \leq \frac{2^4e^4}{\pi^2} 2^{2h+2k}p^{h+k} \sum_{l=2}^{h-1} \left( \frac{q}{p} \right)^l 2^{-4l} \\
& \leq \frac{2^4e^4}{\pi^2}2^{2h+2k}p^{h+k-2}q^2 \sum_{l=2}^{h-1} \frac{1}{2^{4l}} \\
& \leq C2^{2h+2k}p^{h+k-2}q^2.
\end{split}\end{equation*}
Therefore,
\begin{equation*} \begin{split}
& \bigg( \sum_{l=2}^{h-1} 2 \left( h-l \right) \left( k-l \right) (p)_l(q)_l \\
& \qquad \times \sum_{r=0}^{h-l} (p)_{h-l-r}(q)_{r} \tfrac{1}{r+1} \tbinom{h-l}{r} \tbinom{h-l-1}{r} \sum_{s=0}^{k-l} (p)_{k-l-s}(q)_{s} \tfrac{1}{s+1} \tbinom{k-l}{s} \tbinom{k-l-1}{s} \bigg) \\
& \qquad \times \left( O \left( \frac{k^7}{p-k} \right) \right) \\
& \leq C2^{2h+2k}p^{h+k-2}q^2 + \widetilde{C}\frac{2^{2h+2k}k^7p^{h+k-2}q^2}{p-k}.
\end{split}\end{equation*}
Comparing the bounds on each of the three terms completes the lemma.
\end{proof}

\section*{Appendix} 
\label{appendix}

The following theorem gives the limit of the maximum eigenvalues of the sample covariance matrix.
\begin{lemma} [Bai-Silverstein, \cite{BaiSil} ] \label{L: max} Assume that the entries of $\{ x_{ij} \}$ are a double array of iid random variables with mean zero, variance $\sigma^2$, and finite fourth moment. Let $X_p = (x_{ij}; i \leq p, j \leq q)$ be the $p \times q$ matrix of the upper-left corner of the double array. Let $S_p = \frac{1}{p}X_p^TX_p$. If $q/p \to y \in (0, 1]$, then, with probability 1, we have 
\begin{equation*} \begin{split}
\lim_{p \to \infty} \lambda_{max}(S_p) = \sigma^2(1 + \sqrt{y})^2.
\end{split} \end{equation*}
\end{lemma}

For a proof see \cite{BaiSil} Theorem 5.11.

As a direct consequence, we get the following,
\begin{lemma} \label{L: maxew}
Let $\{ p_n: n \geq 1 \}$ and $\{ q_n: n \geq 1 \}$ be two sequences of positive integers such that $q_n \leq p_n$. For each $n$, let $X_n = (x_{ij})$ be a $p_n \times q_n$ matrix where $x_{ij}$ are independent standard Gaussian random variables. If $\lambda_{max}(n)$ denotes the largest eigenvalue of the matrix $X_n^TX_n$, then, with probability 1,
\begin{equation*}
\limsup_{n \to \infty} \frac{\lambda_{max}(n)}{p_n} \leq 4.
\end{equation*}
\end{lemma}

\begin{proof}
If $p_n = q_n$, the lemma follows immediately from Lemma \ref{L: max}. 
If $p_n > q_n$ and $Y$ is an arbitrary $p_n \times (p_n-q_n)$ matrix, then the singular values of $X_n$ are no more than the singular values of the square matrix $(X_n, Y)$. The lemma then follows in general.
\end{proof}

The following uniform version of Stirling's approximation is used in Lemma \ref{L: cov2} and Lemma \ref{L: cov1}.

\begin{lemma} \label{L: stirling} For each positive integer $n$,
\begin{equation*}
\sqrt{2\pi}n^{n+\frac{1}{2}}e^{-n} \leq n! \leq en^{n+\frac{1}{2}}e^{-n}.
\end{equation*}
\end{lemma}

The proof follows from equation (9.15) in \cite{Feller}.

\begin{lemma} \label{L: sum}
Let $p_nq_n = o(n)$ and $l = \frac{\log p_n}{\log (n/p_nq_n)}$. Then $\sum_{j=1}^l \Big( \frac{Cp_nq_n}{n} \Big)^j \frac{1}{j}$ converges to 0 as $n \to \infty$ for any constant $C$.
\end{lemma}

\begin{proof}
Since $\frac{Cpq}{n} \to 0$,
\begin{align*}
\sum_{j=1}^l \Bigg( \frac{Cpq}{n} \Bigg)^j \frac{1}{j} & \leq \sum_{j=1}^{\infty} \Bigg( \frac{Cpq}{n} \Bigg)^j \frac{1}{j} = -\log \left( 1 - \frac{Cpq}{n} \right) \to 0,
\end{align*}
so that $\sum_{j=1}^l \Big( \frac{Cpq}{n} \Big)^j \frac{1}{j}$ converges to 0 as $n \to \infty$ for any constant $C$. 
\end{proof}

\begin{lemma} \label{L: Kn}
Let $p_nq_n = o(n)$. Set 
\begin{equation*}
K_n = \left( \frac{2}{n} \right)^{p_nq_n/2} \prod_{j=1}^{q_n} \frac{\Gamma ((n-j+1)/2)}{\Gamma ((n-p_n-j+1)/2)}. 
\end{equation*}
Then
\begin{equation*}
\log(K_n) = - \frac{p_nq_n^2}{4n} - \sum_{k=1}^{l-1} \frac{p_n^{k+1}q_n}{2(k+1)kn^k} - \frac{p_n^{l+1}q_n}{2ln^l}+ o(1),
\end{equation*}
as $n$ tends to infinity.
\end{lemma}

\begin{proof}
Rewriting, 
\begin{equation*} \begin{split}
\log (K_n) &= \frac{pq}{2} \log \left( \frac{2}{n} \right) + \sum_{j=1}^q  \log \left( \frac{ \Gamma ((n-j+1)/2)}{\Gamma ((n-p-j+1)/2)} \right) 
\end{split} \end{equation*}
Lemma 5.1 in \cite{JiaQi} gives the following
\begin{equation*} \begin{split}
\log (K_n) & = \frac{pq}{2}\log \left( \frac{2}{n} \right)\\
& +\sum_{j=1}^q \bigg[ \frac{n-j+1}{2} \log \left( \frac{n-j+1}{2} \right) \\
& \qquad \quad - \frac{n-p-j+1}{2} \log \left( \frac{n-p-j+1}{2} \right) - \frac{p}{2}\\
& \qquad \quad - \frac{p}{2(n-p-j+1)} + O \left( \frac{p^2 + 4}{(n-p-j+1)^2} \right) \bigg].
\end{split} \end{equation*}
Using the fact that $pq = o(n)$ for the last two terms,
\begin{equation*} \begin{split}
\log (K_n) & = \frac{pq}{2}\log \left( \frac{2}{n} \right)\\
& + \sum_{j=1}^q \bigg[ \left( \frac{n-p-j+1}{2} + \frac{p}{2} \right) \log \left( \frac{n-j+1}{2} \right) \\
& \qquad \quad - \frac{n-p-j+1}{2} \log \left( \frac{n-p-j+1}{2} \right) - \frac{p}{2}  \bigg] + o(1)\\
& = \frac{-pq}{2} + \frac{pq}{2}\log \left( \frac{2}{n} \right) \\
& + \sum_{j=1}^q \bigg[ \frac{n-p-j+1}{2} \log \left( \frac{n-j+1}{n-p-j+1} \right)\\
& \qquad \qquad \qquad + \frac{p}{2} \log \left( \frac{n-j+1}{2} \right) \bigg] + o(1). \\
\end{split} \end{equation*}
Bringing $\frac{pq}{2}\log \left( \frac{2}{n} \right)$ inside the sum yields
\begin{equation*} \begin{split}
\log(K_n)& = \frac{-pq}{2} + \sum_{j=1}^q \bigg[ \frac{n-p-j+1}{2} \log \left( \frac{n-j+1}{n-p-j+1} \right)\\
& \qquad \qquad \qquad \qquad + \frac{p}{2} \log \left( \frac{n-j+1}{n} \right) \bigg] + o(1). \\
\end{split} \end{equation*}
Rewriting the logarithms 
\begin{equation*} \begin{split}
\log \left( \frac{n-j+1}{n-p-j+1} \right) 
& = \log \left( \frac{1 + \frac{p(j-1)}{n(n-p-j+1)}}{1 - \frac{p}{n}} \right) \\
& = - \log \left(1 - \frac{p}{n} \right) + \log \left( 1 + \frac{p(j-1)}{n(n-p-j+1)} \right)
\end{split}\end{equation*}
and 
\begin{equation*}
\log \left( \frac{n-j+1}{n} \right) = \log \left( 1 - \frac{j-1}{n} \right) 
\end{equation*}
yields
\begin{equation*} \begin{split}
\log (K_n) & = \frac{-pq}{2} \\
& + \sum_{j=1}^q \bigg[ -\frac{n-p-j+1}{2} \log \left(1 - \frac{p}{n} \right) \\ 
& \qquad \qquad + \frac{n-p-j+1}{2}\log \left( 1 + \frac{p(j-1)}{n(n-p-j+1)} \right) \\
& \qquad \qquad + \frac{p}{2}\log \left( 1 - \frac{j-1}{n} \right) \bigg] + o(1).
\end{split} \end{equation*}
For the first term
\begin{equation*} \begin{split}
-\log & \left(1 - \frac{p}{n} \right)\sum_{j=1}^q \frac{n-p-j+1}{2} = - \log \left(1 - \frac{p}{n} \right) \left( \frac{(n-p)q}{2} - \frac{q(q-1)}{4} \right). 
\end{split}\end{equation*}
By Taylor's Theorem there exists some $\xi_p \in (0,p)$ such that
\begin{equation*} \begin{split}
- \log \left(1 - \frac{p}{n} \right) & = \sum_{k=1}^l \frac{1}{k} \left( \frac{p}{n} \right)^k + \frac{p^{l+1}}{(l+1)(\xi_p - n)^{l+1}}.
\end{split}\end{equation*}
Now $\xi_p = o(n)$ so the error term in the expansion is 
\begin{equation*} \begin{split}
 \left( \frac{(n-p)q}{2} - \frac{q(q-1)}{4} \right) \left( \frac{p^{l+1}}{(l+1)(\xi_p - n)^{l+1}} \right)  & \leq \frac{p^{l+1}(n-p)q}{(l+1)(\xi_p - n)^{l+1}} \\
& \leq \frac{C p^{l+1}q}{(l+1)(\xi_p - n)^{l}}\\
& = o(1)
\end{split} \end{equation*}
by choice of $l$. (See equation \eqref{eq: 2} for a similar computation.) The first term is thus
\begin{equation*} \begin{split}
& \left( \frac{(n-p)q}{2} - \frac{q(q-1)}{4} \right)\sum_{k=1}^l \frac{1}{k} \left( \frac{p}{n} \right)^k + o(1) \\
& = \sum_{k=1}^l \frac{p^kq}{2kn^{k-1}} - \sum_{k=1}^l \frac{p^{k+1}q}{2kn^k} - \sum_{k=1}^l \frac{q^2}{4k} \left( \frac{p}{n} \right)^k + \sum_{k=1}^l \frac{q}{4k} \left( \frac{p}{n} \right)^k + o(1).
\end{split} \end{equation*}
Notice that
\begin{equation*}\begin{split}
\sum_{k=2}^l \frac{q^2}{4k} & \left( \frac{p}{n} \right)^k + \sum_{k=1}^l \frac{q}{4k} \left( \frac{p}{n} \right)^k \leq \frac{1}{2} \sum_{k=1}^{\infty} \frac{1}{k} \left( \frac{pq}{n} \right)^k  = o(1)
\end{split}\end{equation*}
by Lemma \ref{L: sum}. Therefore the first term is
\begin{equation*} \begin{split}
& \frac{pq}{2} + \sum_{k=1}^{l-1} \frac{p^{k+1}q}{2(k+1)n^k} - \sum_{k=1}^l \frac{p^{k+1}q}{2kn^k} -\frac{pq^2}{4n} + o(1) \\
& = \frac{pq}{2} - \frac{pq^2}{4n} - \sum_{k=1}^{l-1} \frac{p^{k+1}q}{2(k+1)kn^k} - \frac{p^{l+1}q}{2ln^l}+ o(1).
\end{split}\end{equation*}
For the second term there exist $\xi_{p,j} \in \left( 0, \frac{p(j-1)}{n-p-j+1} \right)$ so that
\begin{equation*} \begin{split}
\sum_{j=1}^q  & \frac{n-p-j+1}{2}\log \left( 1 + \frac{p(j-1)}{n(n-p-j+1)} \right) \\
& = \sum_{j=1}^q \frac{n-p-j+1}{2} \bigg[ \sum_{k=1}^l (-1)^{k+1} \frac{1}{k} \left( \frac{p(j-1)}{n(n-p-j+1)} \right)^k \\
& \qquad + \frac{ p^{l+1}(j-1)^{l+1}} {(l+1) (n-p-j+1)^{l+1} (\xi_{p,j} - n)^{l+1}} \bigg]. 
\end{split} \end{equation*}
Notice that
\begin{equation*} \begin{split}
\sum_{j=1}^q \frac{ p^{l+1}(j-1)^{l+1}} {(l+1) (n-p-j+1)^{l+1} (\xi_{p,j} - n)^{l+1}} & = o(1)
\end{split} \end{equation*}
by choice of $l$. The second term is thus
\begin{equation*} \begin{split}
& \sum_{j=1}^q \frac{p(j-1)}{2n} + \sum_{j=1}^q \frac{n-p-j+1}{2} \sum_{k=2}^l (-1)^{k+1} \frac{1}{k} \left( \frac{p(j-1)}{n(n-p-j+1)} \right)^k + o(1) \\
& = \frac{pq^2}{4n} + o(1). 
\end{split} \end{equation*}
For the third term, there exist $\xi_{j-1} \in (0, j-1)$ such that 
\begin{equation*} \begin{split}
\sum_{j=1}^q\frac{p}{2}\log \left( 1 - \frac{j-1}{n} \right) & = - \sum_{j=1}^q \frac{p}{2} \left[ \sum_{k=1}^l \frac{1}{k} \left( \frac{j-1}{n} \right)^k + \frac{(j-1)^{l+1}}{(l+1)(\xi_{j-1} -n)^{l+1}} \right] \\
& = - \sum_{j=1}^q \frac{p}{2} \frac{j-1}{n} - \sum_{j=1}^q\frac{p}{2}\sum_{k=2}^l \frac{1}{k} \left( \frac{j-1}{n} \right)^k \\
& \qquad + \frac{p}{2}\sum_{j=1}^q \frac{(j-1)^{l+1}}{(l+1)(\xi_{j-1} -n)^{l+1}}.
\end{split}\end{equation*}
Now observe that 
\begin{equation*} \begin{split}
\sum_{j=1}^q\frac{p}{2}\sum_{k=2}^l \frac{1}{k} \left( \frac{j-1}{n} \right)^k & \leq \frac{pq}{2} \sum_{k=2}^l \frac{1}{k} \left( \frac{q}{n} \right)^k \\
& \leq \frac{p^2}{2} \sum_{k=2}^l \frac{1}{k} \left( \frac{q}{n} \right)^k \\
& \leq \sum_{k=2}^l \frac{1}{2k} \left( \frac{pq}{n} \right)^k \\
& = o(1)
\end{split} \end{equation*}
and
\begin{equation*} \begin{split}
\frac{p}{2} \sum_{j=1}^q \frac{(j-1)^{l+1}}{(l+1)(\xi_{j-1} -n)^{l+1}} & \leq \frac{p}{2} \sum_{j=1}^q \frac{q^{l+1}}{(l+1)(\xi_{j-1} -n)^{l+1}} \\
& = o(1).
\end{split}\end{equation*} 
Thus the third term is 
\begin{equation*} \begin{split}
\sum_{j=1}^q\frac{p}{2}\log \left( 1 - \frac{j-1}{n} \right) & = \frac{-pq^2}{4n} +o(1).
\end{split} \end{equation*} 
Putting everything together,
\begin{equation*} \begin{split}
\log (K_n) = - \frac{pq^2}{4n} - \sum_{k=1}^{l-1} \frac{p^{k+1}q}{2(k+1)kn^k} - \frac{p^{l+1}q}{2ln^l}+ o(1).
\end{split} \end{equation*}
\end{proof}

Recall that the notation $(x)_a$ represents the  falling factorial, which is defined by
\begin{equation*}
(x)_a = x (x-1)(x-2) \cdots (x-a+1).
\end{equation*}

\begin{lemma} \label{L: Ej}
Let $ \{ p_n : n \geq 1 \}$ and $ \{q_n : n \geq 1 \}$ be two sequences of positive integers such that $p_nq_n = o(n)$ and $q_n \leq p_n$. For each n, let $X_n = (x_{ij})$ be a $p_n \times q_n$ matrix where $x_{ij}$ are independent standard Gaussian random variables. Let 
\begin{equation*}
E_j =  \mathbb{E} \bigg[ \frac{1}{n^j} \bigg( \frac{p_n+q_n+1}{2j} tr(X^TX)^j - \frac{1}{2(j+1)} tr(X^TX)^{j+1} \bigg) \bigg]
\end{equation*}
when $j \leq l-1$ and 
\begin{equation*}
E_l = \mathbb{E} \bigg[ \frac{1}{n^l} \frac{p_n+q_n+1}{2l} tr(X^TX)^l \bigg].
\end{equation*}
Then
\begin{equation*}
\sum_{j=1}^l E_j = \frac{p_nq_n^2}{4n}+\sum_{j=1}^{l-1} \frac{(p_n)_{j}p_nq_n}{2j(j+1)n^j} + \frac{(p_n)_{l}p_nq_n}{2ln^l} + o(1) .
\end{equation*}
\end{lemma}

\begin{proof}
When $j=1$,
\begin{equation*} \begin{split}
E_1&  =  \frac{p+q+1}{2n}pq -  \frac{p^2q + pq^2 - 2pq}{4n} \\
& = \frac{p^2q + pq^2}{4n} + o(1),
\end{split}\end{equation*}
since $pq = o(n)$. When $j=l$: 
\begin{equation*} \begin{split}
E_l & = \frac{p+q+1}{2ln^l} \sum_{s=0}^{l-1} (p)_{l-s}(q)_{s+1}\frac{1}{s+1} \binom{l}{s} \binom{l-1}{s} \\
&= \frac{(p)_{l}pq}{2ln^l} + o(1), 
\end{split} \end{equation*}
by computations similar to equation \eqref{eq:3} in Section 3. For $1 < j < l$,
\begin{align*}
E_j  & = \mathbb{E} \bigg[ \frac{1}{n^j} \bigg( \frac{p+q+1}{2j} tr(X^TX)^j - \frac{1}{2(j+1)} tr(X^TX)^{j+1} \bigg) \bigg] \\
& = \frac{p+q+1}{2jn^j} \mathbb{E} [tr(X^TX)^j] - \frac{1}{2(j+1)n^j} \mathbb{E}[tr(X^TX)^{j+1}] \\
& = \frac{p+q+1}{2jn^j}(p)_jq - \frac{1}{2(j+1)n^j}(p)_{j+1}q+ \frac{p+q+1}{2jn^j}S_1 - \frac{1}{2(j+1)n^j}S_2,  
\end{align*}
where 
\begin{align*}
0 \leq S_1 &  = \sum_{s=1}^{j-1} (p)_{j-s}(q)_{s+1} \frac{1}{s+1} \binom{j}{s} \binom{j-1}{s} \leq \frac{Cp^{j-1}q^22^{2j}}{j^{3/2}}
\end{align*}
 and 
 \begin{align*}
 0 \leq S_2 & = \sum_{s=1}^j (p)_{j+1-s}(q)_{s+1}\frac{1}{s+1}  \binom{j+1}{s} \binom{j}{s}  \leq  \frac{Cp^jq^22^{2j+2}}{(j+1)^{3/2}},
 \end{align*}
by computations similar to equation \eqref{eq:3}. Therefore
\begin{align*}
E_j & = (p)_jq \left( \frac{p+q+1}{2jn^j} - \frac{p-j}{2(j+1)n^j} \right) + \frac{p+q+1}{2jn^j}S_1 - \frac{1}{2(j+1)n^j}S_2  \\
& = \frac{(p)_{j}pq}{2j(j+1)n^j} +  D_j,
\end{align*} 
where 
\begin{align*}
D_j & = (p)_jq \frac{q + 1 + j^2 +qj + j}{2j(j+1)n^j} + \frac{p+q+1}{2jn^j}S_1 - \frac{1}{2(j+1)n^j}S_2,
\end{align*}
and thus
\begin{align*}
|D_j| & \leq \frac{q + 1 + j^2 +qj + j}{2j(j+1)n^j} + \frac{Cp^{j}q^22^{2j}}{j^{5/2}n^j} + \frac{Cp^jq^22^{2j}}{(j+1)^{5/2}n^j}.
\end{align*}
Summing over $1 < j <l$ and using Lemma \ref{L: sum},
\begin{align*}
\sum_{j=2}^{l-1} E_j & = \sum_{j=2}^{l-1} \left( \frac{(p)_{j}pq}{2j(j+1)n^j} + D_j \right) \\
& = \sum_{j=2}^{l-1} \frac{(p)_{j}pq}{2j(j+1)n^j} + o(1).
\end{align*}
\end{proof}

\begin{lemma} \label{L: Cancel}
Let $ \{ p_n : n \geq 1 \}$ and $ \{q_n : n \geq 1 \}$ be two sequences of positive integers such that $p_nq_n = o(n)$ and $q_n \leq p_n$. For each n, let $X_n = (x_{ij})$ be a $p_n \times q_n$ matrix where $x_{ij}$ are independent standard Gaussian random variables. Let 
\begin{equation*}
E_j =  \mathbb{E} \bigg[ \frac{1}{n^j} \bigg( \frac{p_n+q_n+1}{2j} tr(X^TX)^j - \frac{1}{2(j+1)} tr(X^TX)^{j+1} \bigg) \bigg]
\end{equation*}
when $j \leq l-1$ and 
\begin{equation*}
E_l = \mathbb{E} \bigg[ \frac{1}{n^l} \frac{p_n+q_n+1}{2l} tr(X^TX)^l \bigg].
\end{equation*}
Let 
\begin{equation*}
K_n = \left( \frac{2}{n} \right)^{p_nq_n/2} \prod_{j=1}^{q_n} \frac{\Gamma ((n-j+1)/2)}{\Gamma ((n-p_n-j+1)/2)}. 
\end{equation*} Then
\begin{equation*}
\left( \log (K_n) + \sum_{j=1}^l E_j \right) = o(1).
\end{equation*}
\end{lemma}

\begin{proof}
By Lemmas \ref{L: Kn} and \ref{L: Ej}, 
\begin{align*}
\begin{split}
\log (K_n)  & + \sum_{j=1}^l E_j \\
 & = \sum_{j=1}^{l-1} \left( \frac{(p)_jpq}{2j(j+1)n^j} - \frac{p^{j+1}q}{2j(j+1)n^j} \right) \\
 & \qquad + \frac{(p)_l pq}{2ln^l} - \frac{p^{l+1}q}{2ln^l} + o(1).
\end{split} \end{align*}
Expanding the falling factorials, 
\begin{align*} \begin{split}
\sum_{j=1}^{l-1} & \left( \frac{(p)_jpq}{2j(j+1)n^j} - \frac{p^{j+1}q}{2j(j+1)n^j} \right)\\
 &  = \sum_{j=1}^{l-1} \frac{pq^2}{2j(j+1)n^j} \left( (p-1)(p-2) \cdots (p-l+1) - p^{l+1} \right) \\
 & = o(1).
\end{split} \end{align*}
Similarly, 
\begin{equation*}
\frac{(p)_l pq}{2ln^l} - \frac{p^{l+1}q}{2ln^l} = o(1).
\end{equation*}
\end{proof}

\begin{acknowledgements}
The results of this paper are part of a Ph.D. thesis written under the direction of Elizabeth Meckes; the author very much thanks her for many helpful conversations.
\end{acknowledgements}


\begin{thebibliography}{}
%
%
\bibitem{Bai}  Bai, Z.D.: Methodologies in spectral snalysis of large-dimensional random matrices, a review. Statist. Sinica. \textbf{9}, 611-677 (1999) 
\bibitem{BaiSil} Bai, Z.D. and Silverstein, J.W.: Spectral Analysis of Large Dimensional Random Matrices. Science Press, Beijing (2010)
\bibitem{Borel} Borel, E.: Introduction g\'{e}ometrique \'{a} quelques th\'{e}ories physiques. Gauthier-Villars, Paris. (1906)
\bibitem{ChaMec} Chatterjee, S. and  Meckes, E.: Multivariate normal approximation using exchangeable pairs. ALEA. \textbf{4}, 257-283 (2008)
\bibitem{DiaEatLau} Diaconis, P.W., Eaton, M.L. and Lauritzen, S.L.: Finite de Finetti theorems in linear models and multivariate analysis. Scand. J. Statist. \textbf{19}, 289-315 (1992)
\bibitem{DiaFre} Diaconis, P.W. and Freedman, D.: A dozen de Finetti-style results in search of a theory. Ann. Inst. H. Poincare Probab. Statist. \textbf{23}, 397-423 (1987)
\bibitem{Eaton} Eaton, M.L.: Group Invariance Applications in Statistics. IMS, Hayward, CA (1989)
\bibitem{Feller} Feller, W.: An Introduction to Probability Theory and its Applications. Vol. I. Third edition. John Wiley \& Sons, Inc., New York-London-Sydney (1968)
\bibitem{Jiang} Jiang, T.: How many entries of a typical orthogonal matrix can be approximated by independent normals? Ann. Probab. \textbf{34}, 1497-1529 (2006)
\bibitem{JiaMa} Jiang, T. and Ma, Y.: Distance between random orthogonal matrices and independent normals. arXiv:1704.05205, (2017)
\bibitem{JiaQi} Jiang, T.and Qi, Y.: Limiting distributions of likelihood ratio tests for high-dimensional normal distributions. Scandinavian Journal of Statistics. \textbf{42(4)}, 988-1009 (2015)
\bibitem{YinBaiKri} Yin, Y.Q., Bai, Z.D., and Krishnaiah, P.R.: On the limit of the largest eigenvalue of the large-dimensional sample covariance matrix. Probab. Theory Related Fields. \textbf{78}, 509-521 (1988)
\end{thebibliography}


\end{document}